\documentclass[11pt, reqno, oneside]{article}

\usepackage{amsfonts, color}
\usepackage[textwidth=460pt]{geometry}
\usepackage{amstext, amsthm, amssymb, amsmath, mathtools}
\usepackage{mathrsfs}
\usepackage{graphicx}
\usepackage{color}

\usepackage{hyperref}

\usepackage[capitalize,nameinlink,noabbrev]{cleveref}
\hypersetup{colorlinks = true, citecolor = blue}

\usepackage{tcolorbox} % For ERC funding

\crefname{assumption}{Assumption}{Assumptions}
\crefname{definition}{Definition}{Definitions}

\DeclarePairedDelimiter{\tnorm}{\lVert\!\lVert}{\rVert\!\rVert}

\newcommand{\I}{\mathcal{I}}
\newcommand{\J}{\mathcal{J}}
\newcommand{\R}{\mathbb{R}}

\newcommand{\F}{\mathcal{F}}
\newcommand{\U}{\mathcal{U}}

\newcommand{\M}{\mathcal{M}}

\newcommand{\weak}{\rightharpoonup}
\newcommand{\G}{\mathcal{G}}

\newcommand{\NN}{\mathbb{N}}
\newcommand{\xth}{x^\theta}

\newcommand{\Rext}{\mathbb{R}\cup\{+\infty\}}
\newcommand{\Next}{\NN\cup\{+\infty\}}
\newcommand{\diff}{\mathrm{d}}
\newcommand{\cR}{\mathcal{R}}
\newcommand{\FundingLogos}{%
  \raisebox{0pt}{\includegraphics[height=1.5cm]{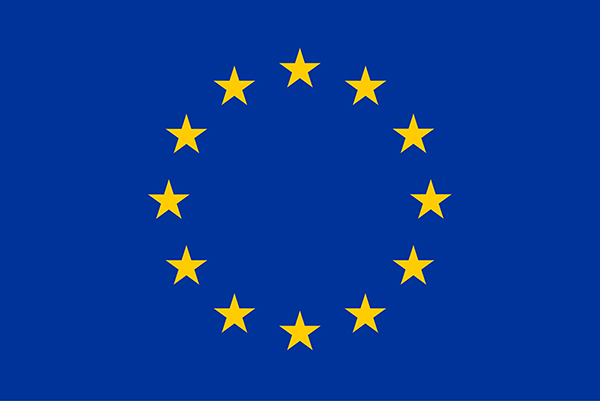}}%
  \hspace{1em}%
  \raisebox{0pt}{\includegraphics[height=1.5cm]{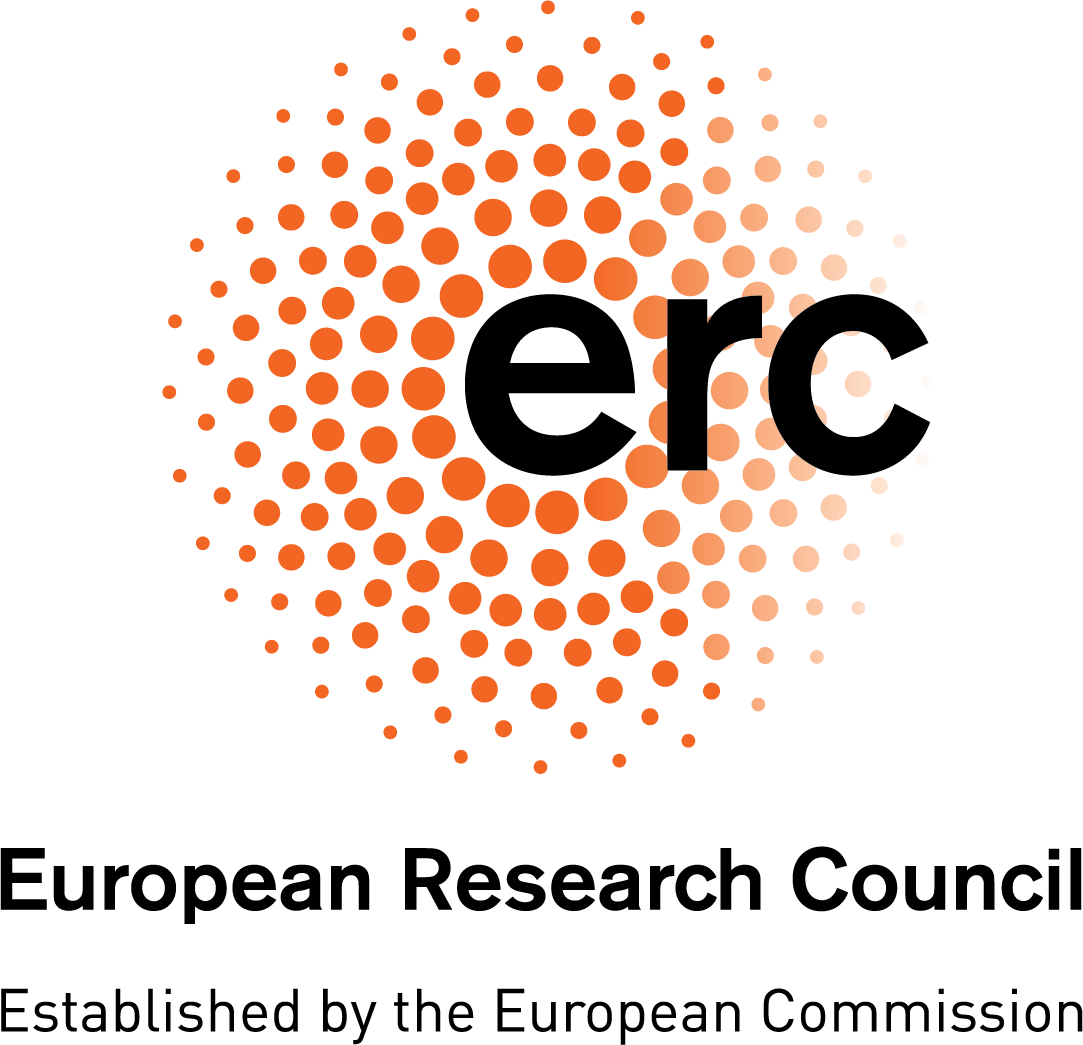}}%
}

\numberwithin{equation}{section}

\theoremstyle{plain}
\begingroup
\newtheorem{theorem}{Theorem}[section]
\newtheorem{lemma}[theorem]{Lemma}
\newtheorem{proposition}[theorem]{Proposition}

\endgroup

\theoremstyle{definition}
\begingroup

\newtheorem{assumption}{Assumption}
\newtheorem{definition}{Definition}
\endgroup

\theoremstyle{remark}
\newtheorem{remark}{Remark}

\title{Risk-Averse Ensemble Control for Control-Affine Systems}

\author{
Alessandro Scagliotti$^{1,2}$ \and
Thomas M.~Surowiec$^{3,4}$ \\
\\
$^1$CIT School, Technical University of Munich, Garching bei München, Germany \\
$^2$Munich Center for Machine Learning (MCML), Munich, Germany \\
$^3$Simula Research Laboratory, Oslo, Norway \\
$^4$SURE-AI Centre for Sustainable, Risk-Averse and Ethical AI, Oslo, Norway
\\
\texttt{scag@ma.tum.de}, \texttt{thomasms@simula.no}
}

\date{}

\begin{document}

\maketitle

\begin{abstract}
    A number of important modern applications in optimal control can be formulated as open loop control problems in which the underlying dynamical systems are subject to random inputs. These so-called \textit{ensemble control} problems require the corresponding optimal control to be deterministic, as it must be computed before the realization of uncertainty and the passage of time. Practical applications of ensemble control include quantum control and the training of Neural ODEs.
    However, the standard approach to ensemble control treats the uncertainty in the objective function via the expectation, which provides optimal controls that only work well on average while ignoring critical outlier phenomena. This study provides a comprehensive mathematical treatment of risk-averse ensemble control. 
    Within this setting, we adopt a control-affine structure that ensures the lower semi-continuity needed for proving the existence of optimal solutions. The central analytical contribution of this paper is a rigorous characterization of the control-to-state mapping in which we establish weak-to-strong continuity, continuous Fréchet differentiability, and weak-to-strong continuity of the derivative operator. Furthermore, this regularity yields primal and dual first-order optimality conditions characterized by an adjoint state of bounded variation, and it fulfills the functional prerequisites required for the convergence of infinite dimensional optimization algorithms. We conclude by validating these theoretical developments through a numerical experiment in quantum control.\\

    \noindent \textbf{Key words: } Simultaneous control, risk-averse optimal control, ensemble optimal control,  control-affine systems, regularity of the control-to-state mapping.\\
    \noindent \textbf{AMS subject classifications: } 49J55, 49K55, 49J27, 90C15, 93C10, 81Q93
\end{abstract}

\section{Introduction}

Ensemble control is a branch of modern optimal control that addresses the challenge of steering parameterized families of dynamical systems using a single broadcast control input \cite{LK06,RuLi12}, \emph{independent of the parameters}. 
As the parametric dependencies in these dynamical systems are typically understood as random inputs, the common approach in the literature is to treat this uncertainty by taking the expectation of the resulting random objective function \cite{BK19,BK24,Scag23,Aronna2025,Aronna2026,melnikov2024convergence}. 
In the parlance of stochastic optimization, this is known as the \textit{risk-neutral} setting, cf.~\cite{Shapiro2021}.

Optimizing solely for the expected value results in optimal controls that will most likely perform well on average, i.e., upon repeated application of the control to out-of-sample realizations of the dynamical system. However, this also permits arbitrarily large performance degradations on scenarios arising from the tail of the distribution of the random inputs. This failure to penalize an additional statistical measure of the random objective function, such as the variance, hinders the establishment of uniform performance guarantees across the parametric ensemble. The necessity for a \textit{risk-averse} paradigm becomes highly evident when examining modern applications such as machine learning and quantum control. Before addressing the literature and further structure of the paper, we briefly detail these two motivating applications.

The training of infinitely deep neural networks can be formalized as an ensemble optimal control problem governed by neural ordinary differential equations \cite{Haber2017,JMLR:v18:17-653,chen2018neural,RuizBalet2021,Hofmann2025} (see also \cite{Palladino2026} for a connection to reinforcement learning). Here, the shared network weights act as the control variables that must simultaneously steer an ensemble of trajectories generated by an underlying data distribution, whose samples provide the initial conditions for the dynamical system and thus constitute the main parametric dependence. The loss function depends the parametrically dependent states at the terminal-time and it measures their deviation from the target observations. 
A parallel challenge arises in quantum control  \cite{BCR10,BST15,ChiGau18}, specifically the quantum transfer with uncertain resonance frequency \cite{RABS}. In contrast to the machine learning example, the parametric uncertainty here enters directly into the governing vector field, representing variations in the natural resonance frequency across the physical ensemble. In this setting, a single broadcast electromagnetic pulse must be designed to robustly steer the quantum states despite this inhomogeneity.

In both applications, many practical computational approaches choose to minimize an empirical tracking cost in the mean square sense. This often stands in stark contrast to what the theory actually demands.
Since universal approximation theorems are typically formulated with error measured in the sup-norm \cite{Cybenko1989,Hornik1991,AS2}, the ideal optimal control should ensure that the resulting neural network approximates the \textit{unknown} continuous mapping uniformly. Similarly, achieving approximate ensemble controllability in the quantum setting requires a strict performance bound evaluated in the sup-norm \cite{Augier2018,RABS,Liang2026}. 
In these contexts, denoting with $\theta$ the parameter affecting the dynamics and/or the initial datum, the prototypical optimal control problems can be formulated as the search for an admissible control trajectory $t\mapsto u^*(t) \in \R^k$ (\emph{independent of the state variable and of the parameter}) such that:
\[
u^* \in \arg\min_{u \in \mathcal{U}} \left\{ \sup_{\theta \in \Theta}  |x^\theta_u (T) - x_{\text{target}}(\theta) |^2 + \frac{\alpha}{2} \int_{[0,T]} |u(t)|^2 \,\diff t \right\}
\]
with $\alpha>0$, and subject to the dynamics $\dot x^\theta_u (t) = F\big( x^\theta_u(t), u(t), \theta \big)$, $\xth_u(0) = x_0(\theta)$ for all $\theta \in \Theta$. 
However, due to the intrinsic non-smoothness of the previous optimization problem, in the practice it is common to replace the supremum and to consider instead objective functions of the type:
\[
u\mapsto \mathbb E_{\theta \sim \mu}\left[|x^\theta_u (T) - x_{\text{target}}(\theta)|^2\right] + \frac{\alpha}{2} \int_{[0,T]} |u(t)|^2 \, \diff t
\]
for some probability measure $\mu$ over $\Theta$. 
Because convergence in the mean square sense does not guarantee uniform bounds, optimal controls derived under standard expectations remain vulnerable to worst-case outliers.  The general method developed in our work allows us to bridge this analytical gap by introducing general convex risk measures $\mathcal{R}$ into the optimal control objective. Hence, this allows for objectives of the form
\[ u\mapsto 
\mathcal{R}_{\theta\sim\mu}\left[ |x^\theta_u (T) - x_{\text{target}}(\theta) |^2\right] + \frac{\alpha}{2} \int_{[0,T]} |u(t)|^2 \, \diff t.
\]
In the case where $\mathcal{R}$ is the Average-Value-at-Risk at confidence level $\beta \in (0,1)$, we can interpolate systematically between computationally tractable average performance ($\beta \to 0$) and the strict uniform bounds demanded by the sup-norm ideal ($\beta \to 1$). The latter is a consequence of the upper quantile function tending to the essential supremum as $\beta \to 1$. The exact effects on the optimal controls are discussed further in detail in \cref{rem:cvar-effect}.
We mention that the class of risk-averse ensemble problems considered in this work is broader than the prototypical examples presented above. Namely, given a risk measure $\mathcal{R}$, we address objective functions of the form
$$ u\mapsto  \mathcal{R}_{\theta\sim\mu}\big[ \J_u(\theta) \big] + \int_{[0,T]} f\big(t,u(t)\big) \,\diff t$$
subject to the control-affine dynamics in $\R^n$ given by $\dot x^\theta_u(t) = F\big( x_u^\theta(t), \theta \big)u(t)$ and $x^\theta(0)=x_0(\theta)$, and where
\begin{equation*} 
    \J_u(\theta) \coloneqq \int_{[0,T]}
    a\big( t, x_u^\theta(t), \theta \big)
    \,\diff \nu(t),
\end{equation*}
with $a\colon [0,T]\times \R^n\times \Theta \to \R$ designing the state-dependent cost, and where $\nu \in \mathcal{M}_+([0,T])$ is a non-negative Radon measure.
Here we mention that necessary optimality conditions for risk-averse stochastic control problems have been derived in \cite{Bonalli2023} by means of set-valued analysis tools.

Although much work has been done in the literature on risk-averse optimal control of random partial differential equations (PDEs) \cite{Surowiec_16,Kouri2018,ks2018existence}, risk-averse ensemble control requires a different set of mathematical tools due to the absence of elliptic partial differential operators. Similar to PDE-constrained optimization, we formulate the optimal control problem by lifting the governing random differential equations to an appropriate Bochner space setting. However, the lack of elliptic operators makes the analysis of the control-to-state mapping highly non-trivial. Within this infinite dimensional setting, we adopt a control-affine structure. On the one hand, this allows us to avoid the need for an analytical relaxation of the control space via Young measures to establish the existence of solutions. On the other hand, ensembles of control-affine systems represent a general and practically relevant class with diverse applications. 
In fact, even in linear-state systems, ensemble control is a source of inspiring questions (see, e.g., \cite{LoZu16} and the more recent contributions \cite{DS21, Dan22, Schoe23}). Moreover, the control-affine setting often appears in quantum problems (see, e.g., \cite{BCR10,BST15,ChiGau18}) and in the control of the Schr\"{o}dinger equation via diffeomorphism approximation (see \cite{Beauchard2025,Pozzoli2025} and \cite{Bartsch2019,Bartsch2021}).
Control-affine systems play a relevant role as well in the mathematical modeling of Deep Learning (see \cite{AS1,AS2,AL24} for universal approximation results).

The central analytical contribution of this paper is a rigorous topological characterization of the control-to-state mapping. Specifically, we establish its weak-to-strong continuity, its continuous Fréchet differentiability, and the weak-to-strong continuity of its derivative operator. These properties allow us to prove the existence of optimal controls for the nonsmooth risk-averse ensemble control problems and to derive primal and dual optimality conditions. The lack of smoothness in the objective functional, induced by the risk measure, partially accounts for the adjoint state lacking absolute continuity in time. Furthermore, the weak-to-strong continuity of the derivative operator is a major prerequisite for the convergence theory of infinite dimensional optimization algorithms \cite{Kouri2021}, providing the necessary analytical foundation for rigorous numerical implementation.

The remainder of this paper is organized as follows. \Cref{sec:prel} establishes the mathematical preliminaries by lifting the parametric ordinary differential equations to the appropriate Bochner space. \Cref{sec:cost-func} formalizes the tracking cost and defines the associated Nemytskij operators. \Cref{sec:risk_av_EOC} introduces the general risk-averse ensemble optimal control problem and proves the existence of minimizers. \Cref{sec:diff-tracking} contains the core topological analysis, detailing the continuous Fréchet differentiability of the control-to-state mapping and the strong convergence of its derivative along weakly convergent sequences. Leveraging these regularity results, \Cref{sec:foopt} derives the primal and dual first-order optimality conditions, explicitly characterizing the adjoint state of bounded variation, and outlines the theoretical prerequisites for algorithmic implementation. Finally, \Cref{sec:num} validates the theoretical framework through a numerical experiment in quantum control, and \Cref{subsec:Nemytskij} provides supporting technical lemmas.

\section{Random Control-Linear Systems} \label{sec:prel}

We commence with a brief study of random control-linear systems. Given a metric space $(\Theta, d)$ equipped with a Borel probability measure $\mu\in \mathcal{P}(\Theta)$, we consider an evolution horizon $[0,T]$ (with $T>0$) and a class of admissible controls $\U \subseteq L^q([0,T],\R^k)$, and we study the following family of control systems in $\R^n$ parametrized by $\theta\in\Theta$: 
\begin{equation}\label{eq:ens_ctrl_sys}
\begin{cases}
\dot x^\theta (t) =  F^\theta \big( \xth(t) \big)u(t)&
\mbox{for a.e. }t\in [0,T],\\
\xth(0) = x_0(\theta),
\end{cases}
\end{equation}
where $F^\theta=(F^\theta_1,\ldots,F^\theta_k)\colon \R^n\to \R^{n\times k}$ define the vector fields of the ensemble of control-linear systems, and $x_0\colon \Theta\to \R^n$ prescribes the Cauchy data for the evolutions.
Using the same notation as in \cite{Scag23}, we introduce $F\colon  \R^n \times \Theta \to\R^{n\times k}$ as:
% \begin{equation}\label{eq:def_fields_maps}
% F(x,\theta) = \big(F_1(x,\theta),\ldots,F_k(x,\theta) \big) \coloneqq \big( F^\theta_1(x),\ldots,F_2^\theta(x) \big)
% \end{equation}
\begin{equation}\label{eq:def_fields_maps}
F(x,\theta) = \big(F_1(x,\theta),\ldots,F_k(x,\theta) \big) \coloneqq \big( F^\theta_1(x),\ldots,F_k^\theta(x) \big)
\end{equation}
for every $\theta\in\Theta$ and $x\in\R^n$. We require $F$ to satisfy the assumption below.

\begin{assumption}[Properties of the Vector Field $F$] \label{ass:fields}
The mapping $F=(F_1,\ldots,F_k)$ defined in \eqref{eq:def_fields_maps} is Borel measurable as functions of $(x,\theta)$ and is globally Lipschitz continuous in the first argument, i.e.,  there exists $L>0$ such that 
\begin{equation*}
\sup_{i=1,\ldots,k} |F_i(x_1,\theta) - F_i(x_2,\theta)| \leq 
L |x_1-x_2|
\end{equation*}
for every $x_1,x_2 \in \R^n$ and for $\mu$-a.e.~$\theta\in \Theta$. Moreover, we require that there exists $C>0$ such that
\begin{equation*}
\sup_{i=1,\ldots,k} |F_i(1,\theta)| \leq C
\end{equation*}
for $\mu$-a.e.~$\theta\in \Theta$.
\end{assumption}

\begin{remark}[Finite-Dimensional Noise]
We observe that \cref{ass:fields} holds if, for example, for every  $i=1,\ldots,k$ the map $F_i\colon \R^n \times \Theta \to \R^n$ can be decomposed as follows:
\begin{equation*}
 F_i(x,\theta) = \tilde F_i (x, \phi(\theta) ),
\end{equation*}
where $\phi\colon \Theta \to \R^l$ is Borel measurable, $\mu\{\theta : |\phi(\theta)|\leq r\}=1$ for some $r>0$, and $\tilde F_i \colon \R^n \times \R^l \to \R^n$ is continuous and satisfies
\begin{equation*}
| \tilde F_i(x_1,z)- \tilde F_i(x_2,z)| \leq \tilde L(z) |x_1-x_2|, \qquad |\tilde F_i(0,z)| \leq \tilde C(z)
\end{equation*}
for every $x_1,x_2 \in \R^n$ and for every $z \in \R^l$,
with $\tilde L, \tilde C \colon \R^l \to (0,+\infty)$ continuous.
\end{remark}

\begin{remark}[Missing Drift Term in \eqref{eq:ens_ctrl_sys}] \label{rmk:control-affine}
    We observe that the control-affine dynamics
    \begin{equation*}
        \dot x^\theta (t) = F^\theta_0 \big( \xth(t) \big) + F^\theta \big( \xth(t) \big)u(t)
    \end{equation*}
    are covered by the framework of \cref{eq:ens_ctrl_sys} if we set $\bar F \coloneqq (F_0,F)\colon \R^n \times \Theta \to\R^{n\times (k+1)}$, $\bar u \coloneqq (1,u)$ for every $u\in \U \subseteq L^q([0,T],\R^k)$, and $\bar \U \coloneqq \{ \bar u\in L^q([0,T],\R^{k+1}) \mid \exists u\in \U : \bar u= (1,u) \}$.
\end{remark}

We also make the following hypothesis on the function $\theta \mapsto x_0(\theta)$.
\begin{assumption}[Integrability of the Initial Value Data]\label{ass:cauchy}
The function $x_0\colon \Theta\to\R^n$ prescribing the initial conditions for the ensemble \eqref{eq:ens_ctrl_sys} is in $L_\mu^{p_0}(\Theta, \R^n)$, with $1<p_0<\infty$.
\end{assumption}

For $\mu$-a.e.~$\theta\in\Theta$, we can write, for every $u\in L^1([0,T],\R^k)$, the curve $x_u^\theta \colon [0,T]\to\R^n$ to denote the solution of the Cauchy problem \eqref{eq:ens_ctrl_sys} corresponding to the system identified by $\theta$ and driven by the control $u$. The existence of such a trajectory follows from the classical Carath\'eodory Theorem (see \cite[Theorem~5.3]{H80}).
We consider $\U \subseteq L^q([0,T],\R^k)$ with $1<q<\infty$ as the space of admissible controls. 
We require $\U$ to be \emph{weakly closed} in $L^q([0,T],\R^k)$, the latter being equipped with the usual Banach space structure.
%Given $u\in\U$, we describe the evolution of the ensemble of control systems \eqref{eq:ens_ctrl_sys} through the mapping $X_u\colon [0,T]\times\Theta\to\R^n$ defined as follows:
%\begin{equation} \label{eq:def_evol_ens}
%X_u(t,\theta) \coloneqq x_u^\theta(t)
%\end{equation}
%for every $t\in[0,T]$ and for $\mu$-a.e.~$\theta\in \Theta$.
Before proceeding, we show that we can lift \eqref{eq:ens_ctrl_sys} to an ODE in the Banach space $L^{p_0}_\mu(\Theta,\R^n)$, which admits a unique solution.
For every $X \in L^{p_0}_\mu(\Theta,\R^n)$ and for every $u\in \U$, we define the mapping $\F_u\colon [0,T]\times L^{p_0}_\mu(\Theta,\R^n) \to L^{p_0}_\mu(\Theta,\R^n)$ as follows:
\begin{equation} \label{eq:def_field_Banach}
    \F_u(t,X)[\theta] \coloneqq 
    \sum_{i=1}^k u_i(t) F_i\big( X(\theta), \theta \big) 
\end{equation}
for a.e.~$t\in[0,T]$ and for $\mu$-a.e.~$\theta\in\Theta$.

\begin{proposition}[Banach Space-Valued Dynamical Systems] \label{prop:existence_ODE_Banach}
    Let \cref{ass:fields} hold. 
    Then, for every $u\in \U \subset L^q([0,T],\R^k)$, the function $\F_u\colon [0,T]\times L^{p_0}_\mu(\Theta,\R^n) \to L^{p_0}_\mu(\Theta,\R^n)$ defined as in \eqref{eq:def_field_Banach} is $L^q$-Carath\'eodory and $L^q$-Lipschitz (see \cref{def:carath}).
    Moreover, under \cref{ass:fields,ass:cauchy}, the Cauchy problem 
    \begin{equation} \label{eq:diff_eq_X}
        \begin{cases}
            \dot X_t^u = \F_u(t,X_t^u) & \mbox{for a.e. }t\in [0,T],\\
            X_0^u = x_0
        \end{cases}
    \end{equation}
    admits a unique solution $X^u\in W^{1,p}\big( [0,T], L^{p_0}_\mu(\Theta,\R^n) \big)$. 
\end{proposition}
We occasionally refer to $X^u$ as the \textit{ensemble trajectory}.
\begin{proof}
    % The fact that $\mathcal{F}_u$ is $L^q$-Carath\'eodory is a direct consequence of its definition and the Lipschitz and boundedness assumptions on the individual vector fields $F_i$.
    We first address the Lipschitz continuity of $\F_u$ in the second argument.
    Owing to \cref{ass:fields}, for every $i=1,\ldots,k$ we have that
    \begin{equation*}
        \left| F_i\big( X_1(\theta), \theta \big) - F_i\big( X_2(\theta), \theta \big) \right| \leq L \left| X_1(\theta) -  X_2(\theta) \right|
    \end{equation*}
    for $\mu$-a.e.~$\theta\in\Theta$, yielding
    \begin{equation*}
        \left\| \F_u\big( t, X_1 \big) - \F_u\big( t, X_2 \big) \right\|_{L^{p_0}_\mu} \leq L( 1 + |u(t)|)\| X_1 - X_2 \|_{L^{p_0}_\mu}
    \end{equation*}
    for a.e.~$t\in [0,T]$. This shows that $X\mapsto \F_u(t,X)$ is $L^q$-Lipschitz continuous at a.e.~time. \\
    Regarding measurability, fix $X\in L^{p_0}_\mu(\Theta,\R^n)$, and consider the mapping
    \begin{equation*}
        t\mapsto \F_u(t,X)[\cdot] = \sum_{i=1}^k u_i(t) F_i\big( X(\cdot), \cdot \big) .
    \end{equation*}
    The time measurability follows as soon as we show that $\theta \mapsto F_i \big(X( \theta), \theta \big) \in L^{p_0}_\mu(\Theta,\R^n)$ for every $i=1,\ldots,k$.
    Recalling that $F_1,\ldots,F_k\colon \R^n\times\Theta\to\R^n$ are continuous in the first argument and Borel measurable in $\theta$, it turns out that $\theta\mapsto F_i \big(X( \theta), \theta \big)$ is measurable, for every $i=0,\ldots,k$. Hence, using the boundedness at the origin provided by \cref{ass:fields}, we deduce that
    \begin{equation*}
        \left\| F_i\big( X(\cdot), \cdot \big) \right\|_{L_\mu^{p_0}} \leq C + L \| X  \|_{L_\mu^{p_0}}
    \end{equation*}
    for every $i=1,\ldots,k$.
    Finally, for every $X\in L^{p_0}_\mu(\Theta,\R^n)$, leveraging on the last estimate, we obtain that
    \begin{equation*}
        \left\| \F_u ( t, X ) \right\|_{L_\mu^{p_0}} \leq \left( C + L\| X  \|_{L_\mu^{p_0}} \right) |u(t)|
    \end{equation*}
    for a.e.~$t\in [0,T]$, and we conclude that $\F_u$ is $L^q$-Carath\'eodory. \\
    The second part of the statement follows directly from \cref{thm:existence_ODE}
\end{proof}

\begin{remark}[Absolutely Continuous Trajectories]
    In this paper, we shall always consider the absolutely continuous representative $X^u\in AC\big( [0,T], L^{p_0}_\mu(\Theta,\R^n) \big)$ of the solution of \eqref{eq:diff_eq_X}.
\end{remark}

Since for $\mu$-a.e.~$\theta\in\Theta$ we can consider, for every $u\in L^q([0,T],\R^k)$, the curve $x_u^\theta \colon [0,T]\to\R^n$ that solves the Cauchy problem \eqref{eq:ens_ctrl_sys}, by virtue of \cref{prop:existence_ODE_Banach} we deduce that
\begin{equation} \label{eq:pointwise_X}
    X_t^u(\theta) = x_u^\theta(t)
\end{equation}
for every $t\in [0,T]$ and for $\mu$-a.e.~$\theta\in\Theta$.

With an argument based on the Gr\"onwall Lemma, we can prove the following result.

\begin{lemma}[Growth Condition on Path Trajectories] \label{lem:bound_traj}
Let \cref{ass:fields,ass:cauchy} hold. Then, for $\mu$-a.e.~$\theta \in \Theta$ and for every $u\in L^q([0,T],\R^k)$
we have 
\begin{equation} \label{eq:bound_X}
\sup_{t\in [0,T]} \big| x^\theta_u(t) \big| \leq    \left( | x_0(\theta) | + C \| u \|_{L^1} \right) e^{L \| u \|_{L^1} }.
\end{equation}
\end{lemma}

\begin{proof}
The proof is classical and makes use of Gr\"onwall's lemma; see \cite[Lemma~A.2]{Scag23}.
\end{proof}

The next convergence result plays a pivotal role in several passages of the paper.

\begin{proposition}[Complete Continuity of Ensemble Trajectories] \label{prop:unif_conv_map_X}
Let \cref{ass:fields,ass:cauchy} hold.
For $\mu$-a.e.~$\theta \in \Theta$ we have the following: 
For every sequence of controls $(u_m)_{m\in\NN} \subset \U$ such that $u_m\weak_{L^q}u_\infty$ as
$m\to\infty$, it holds that
\begin{equation} \label{eq:unif_conv_map_X}
\lim_{m\to\infty}\, \sup_{t\in [0,T]} |x_m^\theta(t) - x_\infty^\theta(t)|
=0 
\end{equation}
where $x_m^\theta\colon [0,T]\to \R^n$ denotes the solution of \eqref{eq:ens_ctrl_sys} driven by $u_m$, for every $m\in \Next$.\\
Moreover, we have that
\begin{equation*}
    \lim_{m\to \infty} \ \sup_{t\in [0,T]}\left \| X^{u_m}_t -X^{u_\infty}_t \right\|_{L^{p_0}_\mu} =0,
\end{equation*}
i.e., the mapping $u \mapsto X^u \in C^0\big( [0,T], L^{p_0}_\mu(\Theta,\R^n) \big)$ is weak-to-strong continuous.
\end{proposition}

\begin{proof}
For the first part of the statement, see \cite[Lemma~7.1]{Scag_GF} for $\U= L^2$. The general case $\U\subset L^q$ with $1<q<\infty$ is analogous. 
For the second part, we observe that
\begin{equation*}
    \sup_{t\in [0,T]}\left \| X^{u_m}_t -X^{u_\infty}_t \right\|_{L^{p_0}_\mu}
    \leq \left\| {\textstyle\sup_{t\in [0,T]}} \, |X^{u_m}_t -X^{u_\infty}_t |\right\|_{L^{p_0}_\mu} =
    \left\| {\textstyle\sup_{t\in [0,T]}} \, |x_{u_m}^\theta(t) -x_{u_\infty}^\theta |\right\|_{L^{p_0}_\mu}.
\end{equation*}
As $\theta \mapsto {\textstyle\sup_{t\in [0,T]}} \, |x_{u_m}^\theta(t) -x_{u_\infty}^\theta |$ converges pointwise to zero by \cref{eq:unif_conv_map_X} and it is dominated by an $L^1$-function owing to \cref{lem:bound_traj}, the assertion follows from the Lebesgue dominated convergence theorem (DCT).
\end{proof}

Although the proofs provided in \cite{Scag23} require the space of admissible controls $\U$ to be $L^2$, those arguments do not make use of the Hilbert space structure of $\U$ and can be \textit{verbatim} generalized to the case of a weakly closed subset $\U \subseteq L^q$ with $1<p<\infty$. 
Indeed, the crucial aspects lie in the reflexivity of $L^q$ and of $W^{1,q}([0,T],\R^n)$ (the natural space for the trajectories solving \eqref{eq:ens_ctrl_sys}), along with the fact that $W^{1,q}([0,T],\R^n)$ is compactly embedded in $C^{0,\alpha}([0,T],\R^n)$ with $\alpha \in  (0,1)$.

\section{Admissible Integrands and Cost Functionals}\label{sec:cost-func}

We now introduce the cost functionals for our optimal control problems. Due to the random nature of the forward problems, the cost functionals must themselves be viewed as a random fields or special classes of parametric superposition operators. As we strive for a certain degree of generality in the presentation, we will require a number of natural data assumptions, which we state throughout this section as needed.

For $u\in \U$ and for a fixed $\theta\in \Theta$, we shall consider a trajectory-cost of the form:
\begin{equation*} 
    \J_u(\theta) = \int_{[0,T]}
    a\big( t, x_u^\theta(t), \theta \big)
    \,\diff \nu(t),
\end{equation*}
where $\nu \in \mathcal{M}([0,T])$ is a finite positive Radon measure on $[0,T]$. We aim to treat in a general and unified way costs of the form
\begin{equation*} 
    \J_u(\theta) = \int_{[0,T]}
    a^{\mathrm{AC}}\big( t, x_u^\theta(t), \theta \big)
    \,\diff t  + \sum_{k=1}^{N_{\mathrm{J}}} a^{\mathrm{J}}_k \big( x_u^\theta(t_k), \theta \big),
\end{equation*}
where $0\leq t_1 <\ldots < t_{N_{\mathrm{J}}}\leq T $ can be seen as ``check-points''. In particular, given any target trajectory $t\mapsto \tilde x(t)$, our framework encompasses the following classical case of quadratic tracking and terminal-state cost:
\begin{equation*} 
    \J_u(\theta) = \int_{[0,T]}
     | x_u^\theta(t) - \tilde x(t)|^2
    \,\diff t  +  a^{T} \big( x_u^\theta(T), \theta \big).
\end{equation*}

We insist on the fact that, in its ultimate formulation, the ensemble optimal control problem does not explicitly depend on the parameter $\theta$, and, as a matter of fact, its solutions---i.e., the optimal controls---are deterministic (independent of $\theta$).

\begin{assumption}[Admissible Integrands $a(t,x,\theta)$] \label{ass:risk_ingred}
    We denote with $\nu \in \mathcal{M}^+([0,T])$ a finite positive Radon measure on $[0,T]$. 
    Assume that $p_0>4$ (see \cref{ass:cauchy}) and let $p_1$ be such that $p_1\in (2, p_0/2]$.\footnote{These strict bounds on the Lebesgue exponents are essential to absorb the loss of integrability incurred when computing the Fréchet differential of the control-to-state mapping, as will be rigorously detailed in \Cref{sec:diff-tracking}. }
    We introduce the function $a\colon [0,T]\times \R^n\times \Theta\to \R$ that satisfies the following:
    \begin{itemize}
        \item (Bounded from Below) There exists $c_a\in \R$ such that $a(t,x,\theta)\geq c_a$ for $\mu$-a.e.~$\theta \in \Theta$ and for every $(t,x)\in [0,T]\times \R^n$;
        \item (Carath\'eodory Property I) For $\mu$-a.e.~$\theta\in \Theta$ we have that $(t,x)\mapsto a(t,x,\theta)$ is continuous, and for every $(t,x)\in [0,T]\times \R^n$ the map $\theta \mapsto a(t,x,\theta)$ is measurable;
        \item (Carath\'eodory Property II) For $\mu$-a.e.~$\theta\in \Theta$ we have that the derivative $(t,x)\mapsto D_x a(t,x,\theta)$ is continuous, and for every $(t,x)\in [0,T]\times \R^n$ the map $\theta \mapsto D_x a(t,x,\theta)$ is measurable;
        \item (Local Lipschitz Continuity) There exists a constant $C^{\mathrm{Lip}}>0$ and a non-negative function $g^{\mathrm{Lip}}\in L^{p_1'}_\mu(\Theta, \R)$ with $p_1'=p_1/(p_1-1)$ such that
        \begin{equation} \label{eq:lip_a}
            |a(t,x,\theta) - a(t,y,\theta)| \leq \big( L(|x|, \theta) + L(|y|, \theta) \big)|x-y|
        \end{equation}
        for $\mu$-a.e.~$\theta\in \Theta$ and for every $(t,x)\in [0,T] \times \R^n$, where $L(r,\theta)\coloneqq C^{\mathrm{Lip}} r^{p_1-1} + g^{\mathrm{Lip}}(\theta)$;
    \item The function defined for $\mu$-a.e.~$\theta\in \Theta$ as
    \begin{equation} \label{eq:bound_a_0}
        \bar g_a(\theta) \coloneqq \sup_{t\in [0,T]} |a(t,0,\theta)|
    \end{equation}
    belongs to $L^1_\mu(\Theta,\R)$;
    \item (Local $C^{1,1}$ Property) There exists a constant $C^{\mathrm{Lip}}_{\mathrm{der}}>0$ and a non-negative function $g^{\mathrm{Lip}}_{\mathrm{der}}\in L^{\frac{p_1}{p_1-2}}_\mu(\Theta, \R)$ such that
        \begin{equation} \label{eq:lip_a_der}
            \left| D_x a(t,x,\theta) - D_x a(t,y,\theta) \right| \leq \big( L_{\mathrm{der}}(|x|, \theta) + L_{\mathrm{der}}(|y|, \theta) \big)|x-y|
        \end{equation}
        for $\mu$-a.e.~$\theta\in \Theta$ and for every $(t,x)\in [0,T] \times \R^n$, where $L_{\mathrm{der}}(r,\theta)\coloneqq C^{\mathrm{Lip}}_{\mathrm{der}} r^{p_1-2} + g^{\mathrm{Lip}}_{\mathrm{der}}(\theta)$;
    \item The function defined for $\mu$-a.e.~$\theta\in \Theta$ as
    \begin{equation} \label{eq:bound_der_a_0}
        \bar g_{D_x a}(\theta) \coloneqq \sup_{t\in [0,T]} |D_x a(t,0,\theta)|
    \end{equation}
    belongs to $L^{p_1'}_\mu(\Theta,\R)$, with $p_1'=p_1/(p_1-1)$.
    \end{itemize}
\end{assumption}

\begin{remark} \label{rmk:bound_a}
    We observe that \cref{ass:risk_ingred} yields a bound on the growth of the function $a$. Namely, by combining \cref{eq:lip_a,eq:bound_a_0}, we deduce that
    \begin{equation} \label{eq:bound_a}
        |a(t,x,\theta)| \leq C^{\mathrm{Lip}}|x|^{p_1} + 2 g^{\mathrm{Lip}}(\theta)|x|+ \bar g_a(\theta)
    \end{equation}
    for $\mu$-a.e.~$\theta\in \Theta$ and for every $(t,x)\in [0,T] \times \R^n$. Moreover, from the continuity in the second variable, it follows that $L^{p_1}_\mu(\Theta,\R) \ni \xi(\cdot) \mapsto a\big( t,\xi(\cdot), \cdot \big) \in L^1_\mu (\Theta,\R)$ is continuous for every $t\in [0,T]$, see, e.g., \cite[Theorem~2.2]{ambrosetti1995primer}.
\end{remark}
The admissible integrands allow us to define the Nemytskij operator $A\colon L^\infty\big([0,T], L^{p_1}_\mu(\Theta,\R^n) \big) \to L^\infty\big([0,T], L^{1}_\mu(\Theta,\R) \big)$ given by:
\begin{equation} \label{eq:def_nemytskij_A}
    \begin{split}
     L^\infty\big([0,T], L^{p_1}_\mu(\Theta,\R^n) \big) \ni   Z \mapsto &A(Z) \in L^\infty\big([0,T], L^{1}_\mu(\Theta,\R) \big), \\
     &A(Z)_t(\cdot) \coloneqq
     a(t, Z_t(\cdot), \cdot) \in L^{1}_\mu(\Theta,\R).
    \end{split}    
\end{equation}
The continuity and differentiability properties of the operator $A$ can be deduced from the results contained in \cite{goldberg1992nemytskij}. However, for the reader's convenience, we provide the proofs within our framework in \cref{subsec:Nemytskij}.\\
For notational convenience, we denote by $\mathcal Y \coloneqq \mathcal L \left( L_\mu^{p_1}(\Theta,\R^n); L_\mu^{1}(\Theta,\R) \right)$ the Banach space of the linear continuous operators from $L_\mu^{p_1}(\Theta,\R^n)$ to $L_\mu^{1}(\Theta,\R)$. Similarly we define the Nemytskij operator $D_x A\colon L^\infty\big([0,T], L^{p_1}_\mu(\Theta,\R^n) \big) \to L^\infty\big([0,T], \mathcal Y)$ as follows:
\begin{equation} \label{eq:def_nemytskij_der_A}
    \begin{split}
     L^\infty\big([0,T], L^{p_1}_\mu(\Theta,\R^n) \big) \ni   Z \mapsto &D_x A(Z) \in L^\infty\big([0,T], \mathcal Y \big), \\
     L_\mu^{p_1}(\Theta,\R^n) \ni \zeta\mapsto &D_x A(Z)_t[\zeta] \coloneqq  D_x a\big(t,Z_t(\cdot),\cdot \big)  \zeta(\cdot)  \in L^1_{\mu}(\Theta,\R), \\
     &D_x A(Z)_t(\cdot) \coloneqq
     D_x a(t, Z_t(\cdot), \cdot) \in L^{p_1'}_\mu(\Theta,\R),
    \end{split}    
\end{equation}
where we recall that $p_1'=p_1/(p_1-1)$ is the conjugate exponent of $p_1$. We prove below that this is in fact the Fr\'echet derivative of A under \Cref{ass:risk_ingred}.
\begin{remark} \label{rmk:cont_diff_a}
    We observe that, for every $\xi \in L_\mu^{p_1}(\Theta, \R^n)$ and for every $t\in [0,T]$, we have that $D_x a(t, \xi(\cdot), \cdot) \in L_\mu^{p_1'}(\Theta,\R^n)$. Indeed, owing to \cref{ass:risk_ingred}, by combining \cref{eq:lip_a_der,eq:bound_der_a_0}, it turns out that
    \begin{equation} \label{eq:bound_der_a}
        |D_x a(t,x,\theta)| \leq C^{\mathrm{Lip}}_{\mathrm{der}} |x|^{p_1-1} + 2 g^{\mathrm{Lip}}_{\mathrm{der}}(\theta)|x|+ \bar g_{D_x a}(\theta)
    \end{equation}
    for $\mu$-a.e.~$\theta\in \Theta$ and for every $(t,x)\in [0,T] \times \R^n$, yielding
    \begin{equation*}
        |D_x a(t,\xi(\theta),\theta)| \leq C^{\mathrm{Lip}}_{\mathrm{der}} |\xi(\theta)|^{p_1-1} + 2 g^{\mathrm{Lip}}_{\mathrm{der}}(\theta)|\xi(\theta)|+ \bar g_{D_x a}(\theta)
    \end{equation*}
    for $\mu$-a.e.~$\theta\in \Theta$, which shows that $D_x a(t, \xi(\cdot), \cdot) \in L_\mu^{p_1'}(\Theta,\R^n)$.
    Using classical results on the continuity of Nemytskij operators (see again \cite[Theorem~2.2]{ambrosetti1995primer}), we deduce that $L^{p_1}_\mu(\Theta,\R) \ni \xi(\cdot) \mapsto D_x a\big( t,\xi(\cdot), \cdot \big) \in L^{p_1'}_\mu (\Theta,\R)$ is continuous for every $t\in [0,T]$, which in turn implies (cf.~\cite[Theorem~1.9]{ambrosetti1995primer}) that $L^{p_1}_\mu(\Theta,\R) \ni \xi(\cdot) \mapsto a\big( t,\xi(\cdot), \cdot \big) \in L^1_\mu (\Theta,\R)$ is Fréchet-differentiable for every $t\in [0,T]$.
\end{remark}
We are now in a position to show that the Nemytskij operator $A$ defined in \cref{eq:def_nemytskij_A} is continuously Fréchet differentiable.

\begin{proposition}[Fr\'echet Differentiability of Admissible Integrands] \label{prop:diff_nemytskij_A}
    Let \cref{ass:risk_ingred} hold. Let $A\colon L^\infty\big([0,T], L^{p_1}_\mu(\Theta,\R^n) \big) \to L^\infty\big([0,T], L^{1}_\mu(\Theta,\R) \big)$ be the Nemytskij operator defined in \cref{eq:def_nemytskij_A}.
    Then, $A$ is continuously Fréchet differentiable, and, given $Z,\zeta \in L^\infty \big( [0,T], L_\mu^{p_1}(\Theta,\R^n) \big)$, the Fréchet derivative
    \begin{equation*}
        A'\colon L^\infty \big( [0,T], L_\mu^{p_1}(\Theta,\R^n) \big) \to \mathcal L \left( L^\infty \big( [0,T], L_\mu^{p_1}(\Theta,\R^n) \big); L^\infty \big( [0,T], L_\mu^{1}(\Theta,\R) \big)
        \right)
    \end{equation*}
    evaluated at the point $Z$ in the direction $\zeta$ is given by
    \begin{equation*}
        \left( A'(Z)[\zeta] \right)_t = D_x A(Z)_t [\zeta_t],
    \end{equation*}
    where the Nemytskij operator $D_x A\colon L^\infty\big([0,T], L^{p_1}_\mu(\Theta,\R^n) \big) \to L^\infty\big([0,T], \mathcal Y)$ is as in \cref{eq:def_nemytskij_der_A}.
\end{proposition}
\begin{proof}
    The proof relies on the application of \cite[Theorem~7]{goldberg1992nemytskij}. 
    The fact that the operator $L^{p_1}_\mu(\Theta,\R) \ni \xi(\cdot) \mapsto a\big( t,\xi(\cdot), \cdot \big) \in L^1_\mu (\Theta,\R)$ is Fréchet-differentiable for every $t\in [0,T]$ has been established in \cref{rmk:cont_diff_a}.
    In addition, we have to show that the operator
    \begin{equation*}
        [0,T]\times L^{p_1}_\mu(\Theta,\R^n) \ni ( t,\xi) \mapsto D_x a \big( t,\xi(\cdot),\cdot \big) \in \mathcal{Y}
    \end{equation*}
    complies with the two following conditions:
    \begin{itemize}
        \item[a)] $t\mapsto D_x a\big( t,\xi(\cdot), \cdot \big) \in \mathcal{Y}$ is Bochner-measurable for every $\xi \in L^{p_1}_\mu(\Theta,\R^n)$;
        \item[b)] $L^{p_1}_\mu(\Theta,\R^n) \ni \xi \mapsto D_x a\big( t,\xi(\cdot),\cdot \big) \in \mathcal{Y}$ is continuous for (almost) every $t\in [0,T]$.
    \end{itemize}
    We recall that $\mathcal Y \coloneqq \mathcal L \left( L_\mu^{p_1}(\Theta,\R^n); L_\mu^{1}(\Theta,\R) \right)$.
    We establish a) by showing that, for every $\xi \in L^{p_1}_\mu(\Theta,\R^n)$, $t\mapsto D_x a\big( t,\xi(\cdot), \cdot \big) \in \mathcal{Y}$ is continuous. Namely, for every $t\in [0,T]$ and for every sequence $(t_n)_n$ such that $t_n \to t$ as $n\to\infty$, we observe that
    \begin{equation*}
        \sup_{\|\zeta\|_{L^{p_1}_\mu} \leq 1} \left\| \Big( D_x a\big(t_n,\xi(\cdot), \cdot) - D_x a\big(t,\xi(\cdot), \cdot) \Big) \zeta(\cdot) \right\|_{L_\mu^1}  = 
        \left\|  D_x a\big(t_n,\xi(\cdot), \cdot) - D_x a\big(t,\xi(\cdot), \cdot)  \right\|_{L_\mu^{p_1'}}
    \end{equation*}
    for every $n\geq 1$. Moreover, by virtue of the continuity of $t\mapsto D_x a \big( t,\xi(\theta),\theta)$ for $\mu$-a.e.~$\theta\in \Theta$ and of the bound in \cref{eq:bound_der_a}, we employ the DCT to deduce that $\left\|  D_x a\big(t_n,\xi(\cdot), \cdot) - D_x a\big(t,\xi(\cdot), \cdot)  \right\|_{L_\mu^{p_1'}} \to 0$ as $n\to\infty$.
    To address b), we note that, in view of \cref{eq:lip_a_der}, we have
    \begin{equation*}
        \begin{split}
        \sup_{\|\zeta\|_{L^{p_1}_\mu} \leq 1} \Big\| \Big( D_x a\big(t, &\xi(\cdot), \cdot)  - D_x a\big(t,\xi'(\cdot), \cdot) \Big) \zeta(\cdot) \Big\|_{L_\mu^1} \\
        & \leq  \sup_{\|\zeta\|_{L^{p_1}_\mu}\leq 1} 
        \left\|   \Big( L_{\mathrm{der}}(|\xi(\cdot)|, \theta) + L_{\mathrm{der}}(|\xi'(\cdot)|, \theta) \Big) |\zeta(\cdot)|   \right\|_{L_\mu^{p_1'}} \|\xi(\cdot) - \xi'(\cdot) \|_{L_\mu^{p_1}},
        \end{split}
    \end{equation*}
    and, after observing that $\big( L_{\mathrm{der}}(|\xi(\cdot)|, \theta) + L_{\mathrm{der}}(|\xi'(\cdot)|, \theta) \big) |\zeta(\cdot)| \in  L_\mu^{p_1'}(\Theta,\R)$, we conclude that b) holds for every $t\in [0,T]$.
    Finally, the last hypotheses required by \cite[Theoremn~7]{goldberg1992nemytskij} is the continuity of the operator $D_x A$ defined in \cref{eq:def_nemytskij_der_A}, and we establish it in \cref{lem:lipschitz_nemytskij_der_A}.
\end{proof}

In order to unburden the notation somewhat, we introduce 
for every $t\in [0,T]$ and for every $u\in\U$ the term $S_t^u\in L^{1}_\mu(\Theta)$ defined by:
\begin{equation} \label{eq:def_Bochner_term}
    S_t^u(\cdot) \coloneqq A(X^u)_t(\cdot),
\end{equation}
where $X^u$ is the ensemble trajectory given by \eqref{eq:diff_eq_X}.
We prove that $t\mapsto S_t^u$ is Bochner-integrable with respect to $\nu \in \mathcal{M}([0,T])$.

\begin{lemma}[Boundedness of the Integral Operator] \label{lem:bochner_operator}
    Let $\nu \in \mathcal{M}([0,T])$ be a positive Radon measure. Then, the operator 
    % \begin{equation} \label{eq:def_integr_op}
    %     \int_{[0,T]} \cdot \ \diff \nu (t) \colon C^0\big([0,T], L_\mu^1(\Theta,\R) \big) \to L_\mu^1(\Theta,\R), \qquad W_\cdot \mapsto   \int_{[0,T]} W_t \, \diff \nu (t)
    % \end{equation}
    \begin{equation} \label{eq:def_integr_op}
       \langle\nu, \cdot\rangle \colon C^0\big([0,T], L_\mu^1(\Theta,\R) \big) \to L_\mu^1(\Theta,\R), \qquad W \mapsto   \langle \nu,W\rangle = \int_{[0,T]} W_t \, \diff \nu (t)
    \end{equation}
    is linear and bounded.
\end{lemma}
\begin{proof}
    Fix $W \in C^0\big([0,T], L_\mu^1(\Theta,\R) \big)$.
    According to \cite[Theorem~3.7.4]{hille1996functional}, the well-posedness of the Bochner integral $\int_{[0,T]} W_t \, \diff \nu (t)$ requires $W$ to be strongly measurable and $\int_{[0,T]} \|W_t\|_{L_\mu^{1}} \,\diff \nu(t)<\infty$. 
    The strong measurability follows from the continuity of $t\mapsto W_t$. On the other hand, we see that
    \begin{equation*}
      \left\|\int_{[0,T]} W_t \,\diff\nu(t) \right\|_{L_\mu^{1}} \leq \int_{[0,T]} \|W_t\|_{L_\mu^{1}} \,\diff \nu(t) \leq |\nu|([0,T])  \sup_{t\in [0,T]} \|W_t\|_{L_\mu^{1}},
    \end{equation*}
    which further implies the boundedness of the operator. Finally, the linearity is immediate.
\end{proof}
%% Original:
% \begin{proof}
%     Fix $W \in C^0\big([0,T], L_\mu^1(\Theta,\R) \big)$.
%     According to \cite[Theorem~3.7.4]{hille1996functional}, the well-posedness of the Bochner integral $\int_{[0,T]} W_t \, \diff \nu (t)$ requires us to show that $W$ is strongly measurable and $\int_{[0,T]} \|W_t\|_{L_\mu^{1}} \,\diff\nu(t)<\infty$. 
%     The strong measurability follows from the continuity of $t\mapsto W_t$. On the other hand, we notice that
%     \begin{equation*}
%       \left\|\int_{[0,T]} W_t \,\diff\nu(t) \right\|_{L_\mu^{1}} \leq \int_{[0,T]} \|W_t\|_{L_\mu^{1}} \,\diff\nu(t) \leq |\nu|([0,T])  \sup_{t\in [0,T]} \|W_t\|_{L_\mu^{1}},
%     \end{equation*}
%     which further implies the boundedness of the operator. Finally, the linearity is immediate.
% \end{proof}

We stress the fact that, in this work, we are mainly interested in positive Radon measures $\nu \in \M^+([0,T])$, as specified in \cref{ass:risk_ingred}. In light of the results above, we can argue that the random objective functionals have sufficient regularity for the existence of solutions and derivation of optimality conditions.

\begin{lemma}[Well-Posedness of the Reduced Objective Functional $\mathcal{J}_u$] \label{lem:Bochner_integr}
    Let \cref{ass:fields,ass:cauchy,ass:risk_ingred} hold.
    For every $u\in\U$, the function $S^u \colon [0,T]\to L^{1}_\mu(\Theta)$ belongs to $C^0\big([0,T], L_\mu^1(\Theta,\R) \big)$.
    In particular, it is Bochner-integrable with respect to the measure $\nu \in \M^+([0,T])$ and 
    \begin{equation} \label{eq:def_tracking_term}
        \J_u \coloneqq \int_{[0,T]} S_t^u \,\diff \nu(t) \in L^{1}_\mu(\Theta).
    \end{equation}
\end{lemma}
\begin{proof}
    To see that $t\mapsto S^u_t \in L^{1}_\mu(\Theta)$ is continuous, fix $t\in [0,T]$ and consider $(t_n)_{n\geq1}$ such that $t_n\to t$. The solution $X^u$ of \eqref{eq:diff_eq_X} is continuous, hence $X_{t_n}^u \to_{L^{p_0}_\mu} X_t^u$ as $n\to \infty$ and, up to the extraction of a non-relabeled subsequence, we deduce that $X_{t_n}^u(\theta) \to X_t^u(\theta)$ for $\mu$-a.e.~$\theta\in \Theta$.
    Hence, for $\mu$-a.e.~$\theta\in \Theta$ we have that
    \begin{equation}
        \lim_{n\to \infty} S_{t_n}^u(\theta) = 
        \lim_{n\to \infty} a \big( t_n, X_{t_n}^u(\theta), \theta \big)
        = a \big( t, X_{t}^u(\theta), \theta \big) = S_t^u(\theta).
    \end{equation}
    To conclude, we apply the DCT. By definition, $S^u_t(\cdot)=A(X^u)_t(\cdot)$ and from \cref{eq:bound_a} it follows that
    \begin{equation*}
    \begin{split} 
        |S_{t_n}^u(\theta)| &\leq  {\textstyle \sup_{\tau \in [0,T]}} \left(
        C^{\mathrm{Lip}}|X_\tau^u(\theta)|^{p_1} + 2 g^{\mathrm{Lip}}(\theta)|X_\tau^u(\theta)| \right)+ \bar g_a(\theta)\\
        &= {\textstyle \sup_{\tau \in [0,T]}} \left(
        C^{\mathrm{Lip}}|x_u^\theta(\tau)|^{p_1} + 2 g^{\mathrm{Lip}}(\theta)|x_u^\theta(\tau)| \right)+ \bar g_a(\theta).
    \end{split}
    \end{equation*}
    Recalling that $\bar g_a \in L_\mu^{1}(\Theta)$ by \cref{ass:risk_ingred}, we are left to show that 
    $$g_0(\theta)\coloneqq {\textstyle \sup_{\tau \in [0,T]}} \left(
        C^{\mathrm{Lip}}|x_u^\theta(\tau)|^{p_1} + 2 g^{\mathrm{Lip}}(\theta)|x_u^\theta(\tau)| \right)$$
    belongs to $L_\mu^{1}(\Theta)$ as well.
    To see this, we observe that, by virtue of \cref{lem:bound_traj}, it turns out that $\theta\mapsto \sup_{\tau \in [0,T]}|x_u^\theta(\tau)| \in L_\mu^{p_0}(\Theta) \subset L_\mu^{p_1}(\Theta)$. We also recall that, as prescribed in \cref{ass:risk_ingred}, $g^{\mathrm{Lip}}\in L^{p_0'}_\mu(\Theta)$ with $p_0'=p_0/(p_0-1)$.
    Therefore, we conclude that the sequence $(S^u_{t_n})_{n\in \NN} \subset L^1_\mu(\Theta)$ is dominated by $g_0 + \bar g_a \in L^1_\mu(\Theta)$, and, by the DCT, we obtain that $S_{t_n}^u \to_{L^1_\mu} S_t^u$ as $n\to \infty$. Finally, we note that we may pass to the full sequence here since the limit $S^u_t$ is unique, i.e., using the Urysohn Lemma.
    %By \cref{eq:bound_a,lem:bound_traj}, there exists a constant $\kappa = \kappa\big( \| x_0 \|_{L_\mu^{p_0}}, \| u \|_{L^1}, T, a \big)$ such that $\| S_t^u \|_{L^{1}_\mu} \leq \kappa$ for every $t\in [0,T]$, and we obtain
    %\begin{equation} \label{eq:est_norm_T}
     %   \|\J_u \|_{L^{1}_\mu} \leq  \int_{[0,T]} \| S_t^u \|_{L_\mu^{1}} \, \diff \nu(t) \leq \kappa \,  \nu([0,T]),
    %\end{equation}
    %which completes the proof.
    The last part of the statement follows directly from \cref{lem:bochner_operator}.
\end{proof}

\begin{remark} \label{rmk:restriction_Nem}
    From the proof of \cref{lem:Bochner_integr} we can distill the following more general conclusion: If we take $Z \in C^0  \big( [0,T], L_\mu^{p_1}(\Theta,\R^n) \big) \subset L^\infty \big( [0,T], L_\mu^{p_1}(\Theta,\R^n) \big)$, then $A(Z) \in C^0 \big( [0,T], L_\mu^{1}(\Theta,\R^n) \big) \subset L^\infty \big( [0,T], L_\mu^{1}(\Theta,\R^n) \big)$. 
\end{remark}

We observe that for $\mu$-a.e.~$\theta\in \Theta$ the mapping $\J \colon \U \to L^{1}_\mu (\Theta)$ satisfies the following identity:
\begin{equation*} 
    \J_u(\theta) = \int_{[0,T]}
    a\big( t, x_u^\theta(t), \theta \big)
    \,\diff \nu(t)
\end{equation*}
for every $u\in \U$.
We show below that $u\mapsto \J_u$ is weak-to-strong sequentially continuous.

\begin{proposition} [Complete Continuity of the Reduced Objective Functional]\label{prop:continuity_tracking}
    Let \cref{ass:fields,ass:cauchy,ass:risk_ingred} hold.
    Then, 
    %for every $u\in \U$, we have that the function $\J_u$ introduced in \eqref{eq:def_tracking_term} belongs to $L^q_{\mu}(\Theta)$.Moreover, 
    for every sequence of controls $(u_m)_{m\in\NN} \subset \U$ such that $u_m\weak_{L^q}u_\infty$ as
    $m\to\infty$, we have 
    \begin{equation} \label{eq:J_strong_conv}
        \lim_{m\to\infty} \|\J_{u_m} - \J_{u_\infty}\|_{L^{1}_\mu} =0.
    \end{equation}
\end{proposition}

\begin{proof}
    Let us consider a sequence of controls $(u_m)_{m\in\NN} \subset \U$ such that $u_m\weak_{L^q}u_\infty$ as
    $m\to\infty$, and for $\mu$-a.e.~$\theta\in \Theta$ let $t\mapsto x^\theta_m(t)$ be the solution of \eqref{eq:ens_ctrl_sys} driven by $u_m$, with $m\in \Next$.
    By virtue of \cref{prop:unif_conv_map_X} and using the continuity of $a$ in $(t,x)$ for $\mu$-a.e.~$\theta\in \Theta$, we obtain that
    \begin{equation*}
        \lim_{m\to\infty} 
        \left\| 
        a\big( \cdot, x_m^\theta(\cdot),\theta \big) - a\big( \cdot, x_\infty^\theta(\cdot),\theta \big)
        \right\|_{C^0([0,T])} = 0
    \end{equation*}
    for $\mu$-a.e.~$\theta\in \Theta$, which implies that 
    \begin{equation} \label{eq:J_pointwise_conv}
        \lim_{m\to\infty} \J_{u_m}(\theta)
        =
        \lim_{m\to\infty}  \left\langle \nu, 
        a\big( \cdot, x_m^\theta(\cdot),\theta \big)
        \right\rangle 
        =  \left\langle \nu, 
        a\big( \cdot, x_\infty^\theta(\cdot),\theta \big)
        \right\rangle 
        = \J_{u_{\infty}}(\theta)
    \end{equation}
    for $\mu$-a.e.~$\theta\in \Theta$.
    % Moreover, arguing as in the proof of \cref{lem:Bochner_integr} (i.e., leveraging on \cref{eq:bound_a,lem:bound_traj}), it is possible to construct a dominant $g_0 + \bar g_a \in L^1_\mu(\Theta)$ for the sequence $(\J_{u_m})_{m\in \NN}\subset L^{ 1}_{\mu}(\Theta)$.
    Moreover, owing to \cref{eq:bound_a} and \cref{lem:bound_traj}, the sequence $(\J_{u_m})_{m\in \NN}\subset L^{ 1}_{\mu}(\Theta)$ is uniformly dominated by the function $g_0(\theta) + \bar g_a(\theta) \in L^1_\mu(\Theta)$. The conclusion in \cref{eq:J_strong_conv} then follows by the DCT.
\end{proof}

\section{Risk-averse ensemble optimal control} \label{sec:risk_av_EOC}

In this section, we provide the formulation of risk-averse ensemble optimal control problems and determine a minimal set of assumptions that guarantee the existence of optimal controls. As discussed in \Cref{sec:cost-func}, the reduced objective functionals are themselves random variables $\mathcal{J}_{u} \in L^1_{\mu}(\Theta)$. The goal of choosing a risk measure is to provide a deterministic scalar surrogate for our risk preference. This ultimately yields an optimal control $u^{\star}$ that shapes the law $(\mu \circ \mathcal{J}_{u^{\star}}^{-1})$ in such a way that we systematically penalize undesirable statistical characteristics of the random cost, thereby hedging against unfavorable outcomes when employing $u^{\star}$ in practice.
% Furthermore, we show the convergence of minimizers corresponding to perturbed risk measures towards the original problem solutions as the perturbation tends to $0$.

At the most abstract level, a risk measure is a functional $\cR \colon L_\mu^1(\Theta)\to \Rext$ that quantifies a specific risk preference. There are many useful classes of risk measures, such as \textit{Coherent Risk Measures} \cite{Artzner1999}, \textit{Convex Risk Measures} \cite{Fllmer2002}, and \textit{Regular Measures of Risk} \cite{Rockafellar2013}, to name a few. For a modern mathematical introduction, we refer the reader to \cite{Shapiro2021} or \cite{Pflug2007}. Finally, we also mention recent developments in \cite{bonalli2025risk_measures}.
Nevertheless, for our purposes, we emphasize that only a subset of the various axioms used to define coherent or regular risk measures is needed to prove the existence of optimal controls. We collect these minimal requirements in the following assumption.

\begin{assumption}[Minimal Properties of the Risk Measure $\cR$] \label{ass:sigma_1}
    The mapping $\cR \colon L_\mu^1(\Theta)\to (-\infty, +\infty]$ is a proper, extended real-valued functional satisfying the following properties:
    \begin{enumerate}
        \item (Convexity) $\cR(\lambda W_1 + (1-\lambda)W_2) \leq \lambda \cR(W_1) + (1-\lambda)\cR(W_2)$ for all $W_1, W_2 \in L_\mu^1(\Theta)$ and $\lambda \in [0,1]$;
        \item (Monotonicity) $\cR(W_1) \geq \cR(W_2)$ whenever $W_1 \geq W_2$ for $\mu$-a.e.~$\theta \in \Theta$;
        \item (Finiteness on Constants) $\cR(c) < +\infty$ for every constant $c \in \mathbb{R}$;
        \item (Lower Semi-Continuity) $\cR$ is lower semi-continuous on $L_\mu^1(\Theta)$;
        \item (Interior Domain) $\mathrm{int}(\mathrm{dom}\,\cR) \neq \emptyset$.
    \end{enumerate}
\end{assumption}

% Having introduced $\cR$, we are in a position to detail the cost related to the performance of the control on the ensemble. More precisely, for every control $u\in \U$, we consider the functional
% \begin{equation} \label{eq:ensemble_cost}
%     u\mapsto \cR(\J_u), 
%     %= \cR \left( \int_{[0,T]} a\big( t, x_u^\theta(t), \theta \big) \, \diff\nu(t) \right)
% \end{equation}
% where the map $u\mapsto \J_u\in L_\mu^1(\Theta)$ has been defined in \eqref{eq:def_tracking_term}.
% Unfortunately, in general we cannot expect the functional in \eqref{eq:ensemble_cost} to admit a minimizer. Indeed, since we do not penalize controls $u$ with ``high energy'', we shall face a lack of coercivity.
% To amend this point, we add an integral running cost for the control, expressed in the form $u\mapsto \int_0^T f\big(t,u(t)\big) \,\diff t$. 

Applying risk measures like those in \Cref{ass:sigma_1} to the objective term $\mathcal{J}_{u}$ defined in \eqref{eq:def_tracking_term} provides us with a scalar-valued deterministic function that we may use in the optimal control problem, i.e. for $u\in \U$, we consider the functional
\begin{equation} \label{eq:ensemble_cost}
    u\mapsto \cR(\J_u), 
    %= \cR \left( \int_{[0,T]} a\big( t, x_u^\theta(t), \theta \big) \, \diff\nu(t) \right).
\end{equation}
However, it is generally not expected that the functional in \eqref{eq:ensemble_cost} will have sufficient coercivity properties need to prove that the sublevel sets are weakly inf-compact in $\mathcal{U}$. Therefore, following \cite{Scag_25}, we introduce a general class of control costs $\rho \colon L^{q}([0,T],\mathbb R^k) \to \mathbb R$ by
\[
\rho(u) = \int_{[0,T]} f\big(t,u(t)\big) \,\diff t,
\]
where $f$  fulfills the conditions listed below.

\begin{assumption}[Properties of the Control Cost] \label{ass:running_cost}
The function $f\colon [0,T]\times \R^k \to \R$ defining the integral cost is such that:
\begin{enumerate}
\item[(i)] $f(\cdot,v)$ is measurable for every $v\in\R^k$.
\item[(ii)] $f(t,\cdot)$ is continuous and convex for a.e.~$t\in [0,T]$.
\item[(iii)] There exist $g^-\in L^1([0,T],\R_+)$, $C'>0$ and $q\in(1,\infty)$ such that 
\begin{equation} \label{eq:growth_f}
f(t,v) \geq C'|v|^q_2 - g^-(t)
\end{equation}
for a.e.~$t\in[0,T]$ and for every $v\in\R^k$.
\item[(iv)] There exist $v'\in\R^k$ and $g^+\in L^1([0,T],\R_+)$ so that $|f(t,v')|\leq g^+(t)$ for a.e.~$t\in [0,T]$.
\item[(v)] There exist $C'' > 0$ and $h \in L^1([0,T], \mathbb{R}_+)$ such that $f(t,v) \leq h(t) + C''|v|_2^q$ for a.e. $t \in [0,T]$ and every $v \in \mathbb{R}^k$.
\end{enumerate}
\end{assumption}
\begin{remark}
Here, we insist on the fact that, when setting the space of controls $\U\subset L^q([0,T],\R^k)$, the exponent $q\in(1,\infty)$ should match with \eqref{eq:growth_f} in \cref{ass:running_cost}.
\end{remark}

For existence of solutions, (v) is not necessary. However, since (i)-(v) ensure that $\rho$ is convex and \textit{continuous} on $L^q$, we are also guaranteed to have directional differentiability/convex subdifferentiability, which we require for the derivation of optimality conditions.

In the remainder of the section, the subject of investigation is the functional $\I\colon \U \to \Rext$ defined as follows:
% \begin{equation} \label{eq:def_functional}
%     \I(u) \coloneqq 
%     \cR(\J_u) + \alpha \int_0^T f\big( t, u(t) \big) \,\diff t\quad (\alpha > 0).
% \end{equation}
\begin{equation} \label{eq:def_functional}
    \I(u) \coloneqq 
    \cR(\J_u) + \alpha \rho(u)\quad (\alpha > 0).
\end{equation}
% for every $u\in \U$, and where $u\mapsto \J_u\in L_\mu^1(\Theta)$ is as in \eqref{eq:def_tracking_term} and $\alpha>0$.
We now establish the existence result for solutions of the optimal control problem related to the minimization of the functional $\I$.

\begin{theorem}[Existence of Optimal Controls] \label{thm:Existence}
    Let \cref{ass:fields,ass:cauchy,ass:risk_ingred,ass:sigma_1,ass:running_cost} hold, and let us consider the functional $\I\colon \U \to \Rext$ introduced in \eqref{eq:def_functional}.
    Then, there exists $u^\star \in \U$ such that $\I (u^\star) = \inf_\U \I$.
\end{theorem}

\begin{proof}
    If $\I \equiv +\infty$, there is nothing to prove.
    Assume then that $\I \not \equiv +\infty$. Hence, there exists $M\in \R$ such that $S_M \coloneqq \{u\in \U : \I(u)\leq M\} \neq \emptyset$.
    If $u\in S_M$, we deduce that
    % \begin{equation} \label{eq:estimate_exist_1}
    %         M\geq \I(u) \geq
    %         \cR(c_a) + \alpha \int_0^T f\big( t, u(t) \big) \, \diff t
    %         \geq \cR(c_a) + \alpha C'\| u \|_{L^q}^q - \alpha\| g^-\|_{L^1},
    % \end{equation}
    \begin{equation} \label{eq:estimate_exist_1}
            M\geq \I(u) \geq
            \cR(c_a) + \alpha \rho(u)
            \geq \cR(c_a) + \alpha C'\| u \|_{L^q}^q - \alpha\| g^-\|_{L^1},
    \end{equation}
    where we use the lower bound $\J_u \geq c_a$ $\mu$-almost everywhere (see \cref{ass:risk_ingred}) and the monotonicity of $\cR$ (see \cref{ass:sigma_1}) in the first inequality, and the lower bound on the integral running cost (see \cref{ass:running_cost}) for the second. Rearranging \eqref{eq:estimate_exist_1}, we deduce the inclusion 
    \begin{equation} \label{eq:estimate_exist_2}
        S_M \subset B_{\varepsilon}(0)\coloneqq 
        \{ u\in \U : \| u \|_{L^q} \leq \varepsilon \}, \qquad \varepsilon\coloneqq \left( 
        \frac{M- \cR(c_a) + \alpha\| g^-\|_{L^1}}{\alpha C'}
        \right)^{\frac{1}{q}}.
    \end{equation}
    Since $q\in (1,+\infty)$, $\I$ is coercive in $L^q$ and thus the sublevel sets are weakly compact in $L^q$.\\
    We now address the weak sequential lower semi-continuity of the functional $\I$. Consider a sequence $(u_m)_{m\geq 1}\subset \U$ such that $u_m \weak_{L^q} u_\infty$ as $m\to\infty$. To this end, we observe that 
    % \begin{equation} \label{eq:lsc_1}
    % \begin{split}
    %     \liminf_{m\to+\infty} \I(u_m)  & = \liminf_{m\to+\infty} \left(
    %     \cR(\J_{u_m}) + \alpha \int_0^t f\big(t,u_m(t)\big) \,\diff t 
    %     \right) \\
    %     &\geq \liminf_{m\to+\infty} 
    %     \cR(\J_{u_m}) + \alpha
    %     \liminf_{m\to+\infty} \int_0^t f\big(t,u_m(t)\big) \,\diff t. 
    % \end{split}
    % \end{equation}
    \begin{equation} \label{eq:lsc_1}
    \begin{split}
        \liminf_{m\to+\infty} \I(u_m)  & = \liminf_{m\to+\infty} \left(
        \cR(\J_{u_m}) + \alpha \rho(u_m) \right) 
        \\
        &\geq \liminf_{m\to+\infty} 
        \cR(\J_{u_m}) + \alpha
        \liminf_{m\to+\infty} \rho(u_m). 
    \end{split}
    \end{equation}
    Since $\left\| \J_{u_m} - \J_{u_\infty} \right\|_{L^1(\Theta)} \to 0$ by \cref{prop:continuity_tracking}, the lower semi-continuity of $\cR$ ensures
    \begin{equation} \label{eq:lsc_risk}
        \liminf_{m\to+\infty} 
        \cR(\J_{u_m}) \geq \cR(\J_{u_{\infty}}).
    \end{equation}
    For the control cost, we leverage classical results on integral functionals (see, e.g., \cite[Thm.~6.54]{FL07}), which given the conditions listed in \cref{ass:running_cost} yield the bound
    \begin{equation} \label{eq:lsc_running}
        \liminf_{m\to+\infty} \rho(u_m) = 
        \liminf_{m\to+\infty} \int_{[0,T]} f\big(t,u_m(t)\big) \,\diff t \geq
        \int_{[0,T]} f\big(t,u_\infty(t)\big) \,\diff t = \rho(u_{\infty}).
    \end{equation}
    Combining \eqref{eq:lsc_1}, \eqref{eq:lsc_risk} and \eqref{eq:lsc_running}, we deduce that $\I$ is sequentially lower semi-continuous with respect to the weak topology of $L^q$. It then follows from the direct method, see \cite[Thm.~1.15]{D93}, that $u_{\infty} = u^{\star} \in \mathcal{U}$ is a minimizer of $\mathcal{I}$ over $\mathcal{U}$.
    % Finally, by virtue of the direct method (see \cite[Thm.~1.15]{D93}), being $\I$ weakly coercive and lower semi-continuous, there exists $u^\star\in\U$ where the minimum is attained.
\end{proof}

\section{Differentiability Properties of the Tracking Term $\mathcal{J}$}\label{sec:diff-tracking}

The derivation of optimality conditions and analysis of derivative-based numerical optimization methods, typically require several differentiability properties of the mapping $\U \ni u \mapsto \J_u$ defined in \cref{eq:def_tracking_term}. In particular, we will show that the derivative mapping $D_{u} \mathcal{J}$ is completely continuous in the base point $u$. Far from a mere technicality, it is in fact a crucial assumption often employed in the convergence proofs of infinite-dimensional optimization algorithms that use iterative smoothing techniques to handle the typically nonsmooth risk measures $\mathcal{R}$, cf. \cite{Kouri2020,Kouri2021}. Unlike PDE-constrained optimization, where such compactness often follows implicitly from the inverse differential operators, establishing it in our setting is considerably more delicate.

In what follows, we understand $\J$ as an operator taking values in $L^1_\mu(\Theta)$, i.e., $\J\colon \U\to L^1_\mu(\Theta)$. We begin by studying the differentiability of the ensemble trajectory mapping $u\mapsto X^u \in C^0 \big( [0,T],L^{p_1}_\mu(\Theta) \big)$.
% , which is defined as the solution of \cref{eq:diff_eq_X} and collects the trajectories of the ensemble (see \cref{eq:pointwise_X}).
We note that the exponent $p_1\in (1,p_0/2]$ should match with the one appearing in the growth condition in \cref{eq:bound_a}. The fact that $p_1$ needs to be bounded from above by $p_0/2$ will be clarified in \cref{rmk:loss_integrab}.
Finally, we stress the fact that \cref{eq:diff_eq_X} admits a unique solution \emph{for every $u\in L^q([0,T], \R^k)$}. For this reason, when discussing the differentiability, we shall consider $L^q([0,T], \R^k) \ni u\mapsto X^u$.

\begin{assumption}[Smoothness Properties of the Vector Field $F$]\label{ass:fields_C1}
    The mappings $F_1,\ldots,F_k \colon \R^n\times \Theta \to \R^n$ defined in \eqref{eq:def_fields_maps} are $C^1$-regular in the first argument, their gradients $\nabla_x F_1,\ldots,\nabla_x F_k \colon \R^n\times \Theta \to \R^n$ are Borel measurable as functions of $(x,\theta)$, and  there exists $L'>0$ (independent of $x$ and $\theta$) such that 
\begin{equation*}
\sup_{i=1,\ldots,k} |\nabla_x F_i(x_1,\theta) - \nabla_x F_i(x_2,\theta)| \leq 
L' |x_1-x_2|
\end{equation*}
for every $x_1,x_2 \in \R^n$ and for $\mu$-a.e.~$\theta\in \Theta$.
\end{assumption}

For every $Y \in L^{p_0}_\mu(\Theta,\R^n)$, for every $u,v\in L^q([0,T],\R^k)$, we define the linearized field $\G_{u,v} \colon [0,T]\times L^{p_0}_\mu(\Theta,\R^n) \to L^{p_0}_\mu(\Theta,\R^n)$ as follows:
\begin{equation} \label{eq:def_diff_field_Banach}
    \G_{u,v} (t,Y)[\theta] \coloneqq 
    \left( 
    \sum_{i=1}^k u_i(t)\nabla_x F_i\big( X_t^u, \theta \big) \right)Y(\theta) 
    + \sum_{i=1}^k v_i(t) F_i\big( X_t^u, \theta \big)
\end{equation}
for a.e.~$t\in[0,T]$ and for $\mu$-a.e.~$\theta\in\Theta$, where $X^u\in W^{1,q}\big( [0,T], L^{p_0}_\mu(\Theta,\R^n) \big)$ is the unique solution of \cref{eq:diff_eq_X}.
We observe that the mapping $\G_{u,v}$ is affine in its second argument. 

\begin{proposition}[The Linearized Ensemble Trajectory] \label{prop:existence_diff_ODE_Banach}
    Let \cref{ass:fields,ass:cauchy,ass:fields_C1} hold. 
    Then, for every for every $u,v\in L^q([0,T],\R^k)$, the function $\G_{u,v}\colon [0,T]\times L^{p_0}_\mu(\Theta,\R^n) \to L^{p_0}_\mu(\Theta,\R^n)$ defined as in \eqref{eq:def_diff_field_Banach} is $L^q$-Carath\'eodory and $L^q$-Lipschitz (see \cref{def:carath}).
    Moreover, the Cauchy problem 
    \begin{equation} \label{eq:diff_eq_Y}
        \begin{cases}
            \dot Y_t^{u,v} = \G_{u,v}(t,Y_t^{u,v}) & \mbox{for a.e. }t\in [0,T],\\
            Y_0^{u,v} \equiv 0
        \end{cases}
    \end{equation}
    admits a unique solution $Y^{u,v}\in W^{1,q}\big( [0,T], L^{p_0}_\mu(\Theta,\R^n) \big)$.
\end{proposition}
\begin{proof}
    The arguments follow the lines of the proof of \cref{prop:existence_ODE_Banach}. 
\end{proof}

\begin{remark}
    In this paper, we shall always consider the absolutely continuous representative $Y^{u,v} \in AC\big( [0,T], L^{p_0}_\mu(\Theta,\R^n) \big)$ of the solution of \eqref{eq:diff_eq_X}.
\end{remark}

Since for $\mu$-a.e.~$\theta\in\Theta$ we can consider, for every $u,v\in L^q([0,T],\R^k)$, the absolutely continuous curve $y_{u,v}^\theta \colon [0,T]\to\R^n$ that solves the linear inhomogeneous Cauchy problem 
\begin{equation} \label{eq:ODE_y}
    \begin{cases}
        \dot y_{u,v}^\theta(t)= \left( \sum_{i=1}^k u_i(t)\nabla_x F_i^\theta\big(x_u^\theta(t)\big)  \right)y_{u,v}^\theta(t) + \sum_{i=1}^k v_i(t)F_i^\theta\big(x_u^\theta(t) \big), \\
        y_{u,v}^\theta(0)= 0,
    \end{cases}
\end{equation}
by virtue of \cref{prop:existence_diff_ODE_Banach} we deduce that
\begin{equation} \label{eq:pointwise_Y}
    Y_t^{u,v}(\theta) = y_{u,v}^\theta(t)
\end{equation}
for every $t\in [0,T]$ and for $\mu$-a.e.~$\theta\in\Theta$.
Finally, owing to classical results (see, e.g., \cite[Theorem~2.2.3]{BressanPiccoli}), we recall that, for  $\mu$-a.e.~$\theta\in\Theta$ and for every $u,v\in L^q([0,T],\R^k)$, we can write $y_{u,v}^\theta$ as follows:
\begin{equation} \label{eq:y_fund_matrix}
    y_{u,v}^\theta(t) = M_u^\theta(t) \int_{[0,t]} M_u^\theta(s)^{-1} \sum_{i=1}^k v_i(s)F_i^\theta\big(x_u^\theta(s)\big) \, \diff s
\end{equation}
for every $t\in [0,T]$, where $t\mapsto M^\theta_u(t) \in \R^{n\times n}$ is the fundamental matrix map and its inverse $t\mapsto M_u^\theta(t)^{-1} \in \R^{n\times n}$ are matrix-valued curves that solve, respectively,
\begin{equation} \label{eq:fund_matrix}
    \begin{cases}
        \frac{\diff}{\diff t}M^\theta_u(t) = \left(  \sum_{i=1}^k u_i(t)\nabla_x F_i^\theta \big(x_u^\theta(t)\big) \right) M^\theta_u(t),\\
        M_u^\theta(0) = \mathrm{Id},
    \end{cases}
\end{equation}
and
\begin{equation} \label{eq:fund_matrix_inv}
    \begin{cases}
        \frac{\diff}{\diff t}M^\theta_u(t)^{-1} = -M^\theta_u(t)^{-1} \left(  \sum_{i=1}^k u_i(t)\nabla_x F_i^\theta \big(x_u^\theta(t)\big) \right) ,\\
        M_u^\theta(0)^{-1} = \mathrm{Id}.
    \end{cases}
\end{equation}
\begin{remark} \label{rmk:linear_Y}
    By combining \cref{eq:pointwise_Y,eq:y_fund_matrix}, we observe that $L^q([0,T],\R^k) \ni v \mapsto Y^{u,v}$ is linear. The boundedness of this operator can be deduced from \cref{lem:bound_traj,ass:fields_C1} with Gr\"onwall-like estimates. See later in \cref{eq:bound_fund_mat} in the proof of \cref{lem:conv_fund_mat} for further details.
\end{remark}

The first step consist in studying the differentiability of $u\mapsto x_u^\theta(t)$ when $\theta \in \Theta$ is fixed. We observe that, in what follows, we do not restrict the base point $u$ to vary in $\U$, but we establish the differentiability in the whole $L^q([0,T],\R^k)$.

\begin{lemma} \label{lem:traj_diff}
    Let \cref{ass:fields,ass:cauchy,ass:fields_C1} hold. 
    Then, for $\mu$-a.e. $\theta\in \Theta$ and for every $u,v\in L^q([0,T],\R^k)$ such that $\|v\|_{L^q}\leq 1$, we have that
    \begin{equation} \label{eq:traj_diff}
        \sup_{t\in [0,T]}|x_{u+v}^\theta(t) - x_u^\theta(t) - y_{u,v}^\theta(t)| 
        \leq \kappa_u(\theta) \|v\|^2_{L^1},
    \end{equation}
    for every $t\in [0,T]$, where
    \begin{equation} \label{eq:def_kappa}
    \begin{split}
        \kappa_u(\theta) \coloneqq 
        \bigg( 
        \frac{L'}2
&\left( C + L   
            \Big( | x_0(\theta) | + C \| u \|_{L^1}  \Big)
            \right)^2 
            e^{4L \| u \|_{L^1}} 
             \|u\|_{L^1}\\
            & \qquad +  L \left( C + L   
            \Big( | x_0(\theta) | + C \| u \|_{L^1}  \Big)
            \right) 
            e^{2L \| u \|_{L^1}}
        \bigg) e^{L'\|u\|_{L^1}}.
    \end{split}
\end{equation}
\end{lemma}
\begin{proof}
    Let us fix $\theta \in \Theta$, $u \in \U$ and $v\in L^q([0,T], \R^k)$ with $\|v\|_{L^q} \leq 1$, and let us define $t \mapsto \xi_{u,v}^\theta(t)$ as $\xi_{u,v}^\theta(t) \coloneqq x_{u+v}^\theta(t) - x_u^\theta(t) - y_{u,v}^\theta(t)$, where $y_{u,v}^\theta$ solves the linear system in \cref{eq:ODE_y}. Then using $\xi^{\theta}_{u,v}(t) = \int_{0}^t\dot{\xi}^{\theta}_{u,v}(s) \diff s$ along with the definitions of $x^{\theta}_{u+v}, x^{\theta}_u$ and $y^{\theta}_{u,v}$, we compute
    \begin{equation*}
        \begin{split}
            \big| \xi_{u,v}^\theta(t) \big| & \leq 
            \int_{[0,t]} 
            \Bigg|  \sum_{i=1}^k \Big( F_i^\theta\big(x_{u+v}^\theta(s)\big) - F_i^\theta\big(x_{u}^\theta(s)\big) \Big) u_i(s)
             - \sum_{i=1}^k \nabla_x F_i^\theta\big( x_u^\theta(t) \big) u_i(s) y_{u,v}^\theta(s) \\
            & \qquad\quad
            + \sum_{i=1}^k \Big( F_i^\theta\big(x_{u+v}^\theta(s)\big) - F_i^\theta\big(x_{u}^\theta(s)\big) \Big) v_i(s)
            \Bigg|\,\diff s,
        \end{split}
    \end{equation*}
which, by adding and subtracting $\nabla_x F_i^\theta\big( x_u^\theta(s) \big) \big( x_{u+v}^\theta(s) - x_{u}^\theta(s) \big)$ for $i=1,\ldots,k$, yields
    \begin{equation} \label{eq:comput_xi}
        \begin{split}
            \big| \xi_{u,v}^\theta(t) \big| & \leq 
             \int_{[0,t]} \sum_{i=1}^k \Big| F_i^\theta\big(x_{u+v}^\theta(s)\big) - F_i^\theta\big(x_{u}^\theta(s)\big) - \nabla_x F_i^\theta\big( x_u^\theta(s) \big) \big( x_{u+v}^\theta(s) - x_{u}^\theta(s) \big) \Big| |u_i(s)| \,\diff s \\
& \qquad + \int_{[0,t]}  \sum_{i=1}^k \Big| F_i^\theta\big(x_{u+v}^\theta(s)\big) - F_i^\theta\big(x_{u}^\theta(s)\big) \Big| |v_i(s)|
\,\diff s  \\
& \qquad + \int_{[0,t]}  \sum_{i=1}^k   \Big| \nabla_x F_i^\theta\big( x_u^\theta(s) \big) \Big||u_i(s)| \big|\xi_{u,v}^\theta(s)\big| \,\diff s
        \end{split}
    \end{equation}
for every $t\in [0,T]$. Owing to \cref{ass:fields_C1}, we observe that for every $i=1,\ldots,k$ we have
\begin{equation*}
\begin{split}
\Big| F_i^\theta\big(x_{u+v}^\theta(s)\big)& -  F_i^\theta\big(x_{u}^\theta(s)\big)- \nabla_x F_i^\theta\big( x_u^\theta(s) \big) \big( x_{u+v}^\theta(s) - x_{u}^\theta(s) \big) \Big| \\
& \leq 
\int_0^1 \Big| \nabla_x F_i^\theta\big( x_u^\theta(s) + \tau(x_{u+v}^\theta(s) - x_{u}^\theta(s)) \big) -  \nabla_x F_i^\theta\big( x_u^\theta(s) \big)  \Big| \big| x_{u+v}^\theta(s) - x_{u}^\theta(s) \big| \,\diff \tau \\
& \leq \frac{L'}2  \big| x_{u+v}^\theta(s) - x_{u}^\theta(s) \big|^2
\end{split}
\end{equation*}
for every $s\in [0,T]$, so that, by virtue of \cref{lem:lipsch_traj}, we can bound the first term on the right-hand side of \cref{eq:comput_xi} by
\begin{equation} \label{eq:comput_xi_1}
\begin{split}
 \int_{[0,t]} \sum_{i=1}^k &\Big| F_i^\theta\big(x_{u+v}^\theta(s)\big) - F_i^\theta\big(x_{u}^\theta(s)\big) - \nabla_x F_i^\theta\big( x_u^\theta(s) \big) \big( x_{u+v}^\theta(s) - x_{u}^\theta(s) \big) \Big| |u_i(s)| \,\diff s\\
&  \leq \frac{L'}2
\left( C + L   
            \Big( | x_0(\theta) | + C \| u \|_{L^1}  \Big)
            \right)^2 
            e^{4L \| u \|_{L^1}} \|v\|_{L^1}
            \int_{[0,t]} \sum_{i=1}^k |u_i(s)| \,\diff s  
\end{split}
\end{equation}
for every $t\in [0,T]$. Moreover, owing to \cref{ass:fields} and \cref{lem:lipsch_traj}, we estimate the second integral on the right-hand side of \cref{eq:comput_xi} with
\begin{equation} \label{eq:comput_xi_2}
\begin{split}
\int_{[0,t]}  \sum_{i=1}^k &\Big| F_i^\theta\big(x_{u+v}^\theta(s)\big) - F_i^\theta\big(x_{u}^\theta(s)\big) \Big| |v_i(s)|
\,\diff s \\
& \leq 
L \left( C + L   
            \Big( | x_0(\theta) | + C \| u \|_{L^1}  \Big)
            \right) 
            e^{2L \| u \|_{L^1}} \|v\|_{L^1} \int_{[0,t]} \sum_{i=1}^k |v_i(s)| \, \diff s    
\end{split}
\end{equation}
for every $t\in [0,T]$.
Finally, using \cref{eq:comput_xi,eq:comput_xi_1,eq:comput_xi_2} and \cref{ass:fields_C1}, Gr\"onwall Lemma yields the statement.
\end{proof}
\begin{remark} \label{rmk:loss_integrab}
    We observe that in \cref{lem:traj_diff} the constant $\kappa$ that appears on the right-hand side of \cref{eq:traj_diff} satisfies $\kappa_u(\theta) \simeq |x_0(\theta)|^2$ for $\mu$-a.e.~$\theta \in \Theta$. Therefore, we deduce that $\kappa(\cdot) \in L^{p_0/2}_\mu(\Theta)$. 
    % This loss of integrability when computing the differential with respect to the control is to some extent expected. Indeed, similar phenomena occur in the case of controlled PDEs (see, e.g., CITE Here).
\end{remark} 

We are in a position to prove the Fr\'echet differentiability of the mapping $\U \ni u\mapsto X_\cdot^u \in C^0 \big( [0,T],L^{p_1}_\mu(\Theta) \big)$ for $p_1\leq p_0/2$.

\begin{proposition}[Differentiability of the Ensemble Control-to-State Mapping] \label{prop:diff_Frechet}
    Let \cref{ass:fields,ass:cauchy,ass:fields_C1} hold. For every $u\in L^q([0,T],\R^k)$, let $X_\cdot^u \in C^0 \big( [0,T],L^{p_0}_\mu(\Theta) \big)$ be the solution of \cref{eq:diff_eq_X}, which collects the trajectories of the ensemble (see \cref{eq:pointwise_X}).
    Then, if in \cref{ass:cauchy} we have $p_0\geq 2$ and if we consider $p_1\in (1, p_0/2]$, then $\U \ni u\mapsto X^u \in C^0 \big( [0,T],L^{p_1}_\mu(\Theta) \big)$ is Fr\'echet differentiable at every point $u\in \U$, and 
    \begin{equation} \label{eq:diff_Frechet}
        \sup_{t\in[0,T]} \left\| X^{u+v}_t(\cdot) -   X^{u}_t(\cdot)
        - Y^{u,v}_t(\cdot)\right\|_{L^{p_1}_\mu} \leq 
        \left\| \kappa_u(\cdot) \right\|_{L^{p_1}_\mu} \|v\|_{L^1}^2 
    \end{equation}
    for every $v \in L^q([0,T],\R^k)$ with $\| v\|_{L^q}\leq 1$, where $\kappa_u(\cdot) \in L^{p_0/2}_\mu(\Theta)$ is defined as in  \cref{eq:def_kappa}.
\end{proposition}
\begin{proof}
    Owing to the estimate provided by \cref{lem:traj_diff} for $\mu$-a.e.~$\theta\in \Theta$, it turns out that
        \begin{equation*}
        \begin{split}
        \sup_{t\in[0,T]} \left\| X^{u+v}_t(\cdot) -   X^{u}_t(\cdot)
        - Y^{u,v}_t(\cdot)\right\|_{L^{p_1}_\mu} &\leq 
         \left\| \sup_{t\in[0,T]} \Big| X^{u+v}_t(\cdot) -   X^{u}_t(\cdot)
        - Y^{u,v}_t(\cdot)
        \Big|\right\|_{L^{p_1}_\mu} \\
        &\leq 
        \left\| \kappa_u(\cdot) \right\|_{L^{p_1}_\mu} \|v\|_{L^1}^2 ,
        \end{split}
    \end{equation*}
    and this concludes the proof.
\end{proof}

From \cref{eq:diff_Frechet} we read that, when we evaluate the Fr\'echet differential $D_u X^{u}$ at $u$ in the direction $v$, we obtain $D_u X^{u}[v] = Y^{u,v}$.
The next step consists in showing that, for every sequence $(u_m)_m$ such that $u_m\weak_{L^q}u$ as $m\to\infty$, the sequence of the Fr\'echet differentials $(D_u X^{u_m})_m$ converges strongly to $D_u X^{u}$. 
%i.e., 
%\begin{equation}
%    \lim_{m\to \infty} 
    %\sup_{\|v\|_{L^q}\leq 1} \sup_{t\in[0,T]} \left\| D_u X_t^{u_n}[v] - D_u X_t^{u}[v] \right\|_{L^{p_1}_\mu}  =
%    \sup_{\|v\|_{L^q}\leq 1} \sup_{t\in[0,T]} \left\| Y^{u_m,v}_t(\cdot) - Y^{u,v}_t(\cdot) \right\|_{L^{p_1}_\mu} =0.
%\end{equation}
As done before for the differentiability, it is convenient to first study the convergence when $\theta$ is fixed.
To this end, we need an auxiliary lemma.

\begin{lemma} \label{lem:conv_fund_mat}
    Let \cref{ass:fields,ass:cauchy,ass:fields_C1} hold. 
    For $\mu$-a.e.~$\theta \in \Theta$ we have the following: 
    For every sequence of controls $(u_m)_{m}$ such that $u_m\weak_{L^q}u_\infty$ as $m\to\infty$, then 
    \begin{equation*}
        \lim_{m\to \infty} \sup_{t\in [0,T]} \big|M_{u_m}^\theta(t) - M_{u_\infty}^\theta(t)  \big| =
        \lim_{m\to \infty} \sup_{t\in [0,T]} \big|M_{u_m}^\theta(t)^{-1} - M_{u_\infty}^\theta(t)^{-1}  \big| = 0,
    \end{equation*}
    where $t\mapsto M_{u}^\theta(t)$ and $t\mapsto M_{u}^\theta(t)^{-1}$ solve \cref{eq:fund_matrix,eq:fund_matrix_inv}.
\end{lemma}
\begin{proof}
    We detail the proof only for the convergence of $\big( M_{u_m}^\theta \big)_m$, as the other one follows from the same arguments. By virtue of \cref{ass:fields}, from the Gr\"onwall Lemma we obtain that 
    \begin{equation} \label{eq:bound_fund_mat}
        \big| M_{u}^\theta(t) \big| \leq e^{L k \|u\|_{L^1}}
    \end{equation}
    for every $u\in L^q([0,T],\R^k)$. 
    Moreover, recalling that $\big|\frac{\diff}{\diff t}M^\theta_u(t)\big| \leq kL e^{L k \|u\|_{L^1}}|u(t)|$ and using the mean value theorem, we can show that $t\mapsto M_{u}^\theta(t)$ is $(1-\frac1q)$-H\"older continuous, i.e.,
    \begin{equation} \label{eq:hoelder_fund_mat}
        \big| M_{u}^\theta(t) - M_{u}^\theta(s) \big|
        \leq kL e^{L k \|u\|_{L^1}} \|u\|_{L^q} |t-s|^{1-\frac1q}
    \end{equation}
    for every $u\in L^q([0,T],\R^k)$ and for every $s,t\in [0,T]$.\\
    We prove the statement by arguing with the Urysohn Lemma: For every subsequence $(M^\theta_{u_{m_j}})_j$, there exists a sub-subsequence (not relabelled) such that $M^\theta_{u_{m_j}} \to_{C^0} M^\theta_{u_\infty}$ as $j\to \infty$.
    Since $(u_{m_j})_j$ is a bounded sequence in $L^q$, owing to \cref{eq:bound_fund_mat,eq:hoelder_fund_mat} we apply Arzel\`a-Ascoli Theorem to deduce that the sequence $(M^\theta_{u_{m_j}})_j$ is pre-compact in the $C^0$-norm. Therefore, we consider a sub-subsequence such that $M^\theta_{u_{m_j}} \to_{C^0} \tilde M^\theta$ as $j\to \infty$. We are left to show that $\tilde M^\theta \equiv M^\theta_{u_\infty}$.
    To this end, by rewriting \cref{eq:fund_matrix} as an integral equation, we observe that
    \begin{equation*}
        \begin{split}
            M^\theta_{u_{m_j}}(t) &= \mathrm{Id} + 
            \int_{[0,t]}
            \left(  \sum_{i=1}^k u_{m_j, i}(s) \nabla_x F_i^\theta\big(x_{u_{m_j}}^\theta(s) \big) \right) M^\theta_{u_{m_j}}(s)
            \,\diff s
        \end{split}
    \end{equation*}
    for every $t\in [0,T]$. Moreover, leveraging on \cref{prop:unif_conv_map_X} and on \cref{ass:fields_C1}, and recalling that $M^\theta_{u_{m_j}} \to_{C^0} \tilde M^\theta$ and $u_{m_j}\weak_{L^q}u_\infty$ as $j\to \infty$, we can pass to the limit on the two sides of the last identity, yielding
    \begin{equation*}
        \begin{split}
            \tilde M^\theta(t) &= \mathrm{Id} + 
            \int_{[0,t]}
            \left( \sum_{i=1}^k u_{\infty, i}(s) \nabla_x F_i^\theta\big(x_{u_{\infty}}^\theta(s) \big)  \right)  \tilde M^\theta(s)
            \,\diff s
        \end{split}
    \end{equation*}
    for every $t\in [0,T]$. This implies that $\tilde M^\theta \equiv M^\theta_{u_\infty}$ on $[0,T]$, as they both solve \cref{eq:bound_fund_mat} with driving control $u_\infty$, and we conclude the proof.
\end{proof}

We are finally in a position to address the main result of this section, which is built on the following result. This establishes the compactness of the derivative of the ensemble control-to-state mapping.  In contrast to optimal control of partial differential equations, which rely on compactness properties of embeddings between Sobolev spaces and the inverse of elliptic operators, we need entirely different arguments to prove the main statement.

\begin{proposition}[Compactness of the Derivative of the Ensemble Control-to-State Mapping] \label{prop:strong_conv_diff}
    Let \cref{ass:fields,ass:cauchy,ass:fields_C1} hold. Let us assume that in \cref{ass:cauchy} we have $p_0> 2$, and consider $p_1\in (1, p_0/2]$. 
    For every $u,v\in L^q([0,T],\R^k)$, let $D_u X^{u}[v] \in C^0 \big( [0,T],L^{p_1}_\mu(\Theta) \big)$ be the Fr\'echet differential of the mapping $u\mapsto X^u \in C^0 \big( [0,T],L^{p_1}_\mu(\Theta) \big)$ computed at the point $u\in L^q([0,T],\R^k)$ and evaluated in the direction $v\in L^q([0,T],\R^k)$. Then, for every sequence of controls $(u_m)_{m}$ such that $u_m\weak_{L^q}u_\infty$ as $m\to\infty$, we have
    \begin{equation*}
    \lim_{m\to \infty} \, 
    \sup_{\|v\|_{L^q}\leq 1}  \,  \sup_{t\in[0,T]} \big\| D_u X_t^{u_m}[v] - D_u X_t^{u_\infty}[v] \big\|_{L^{p_1}_\mu} =0.
    \end{equation*}
    In particular, the mapping $u\mapsto X^u$ is continuously Fréchet differentiable.
\end{proposition}
\begin{proof}
    From \cref{eq:diff_Frechet} we read that $D_u X_\cdot^{u}[v] = Y_\cdot^{u,v}$ for every $u,v\in L^q([0,T],\R^k)$, so that we can rephrase the statement as follows:
    \begin{equation*} %\label{eq:strong_conv_diff}
        \lim_{m\to \infty} \, 
    \sup_{\|v\|_{L^q}\leq 1}  \,
    \sup_{t\in[0,T]} \big\| Y^{u_m,v}_t(\cdot) - Y^{u_\infty,v}_t(\cdot) \big\|_{L^{p_1}_\mu} =0.
    \end{equation*}
    Moreover, for $\mu$-a.e.~$\theta \in \Theta$, we have that $Y^{u_m,v}_t(\theta) = y_{u_m,v}^\theta(t)$ for every $m\in \NN\cup \{\infty\}$, for every $v\in L^q([0,T],\R^k)$ and for every $t\in [0,T]$. Hence, by using \eqref{eq:y_fund_matrix} we compute
    \begin{equation} \label{eq:comput_y_conv}
        \begin{split}
            \big| y_{u_m,v}^\theta(t) - & y_{u_\infty,v}^\theta(t) \big|  = \bigg| M_{u_m}^\theta(t) \int_{[0,t]}  M_{u_m}^\theta(s)^{-1} \sum_{i=1}^k F_i^\theta \big(x_{u_m}(s) \big)v_i(s)
            \,\diff s \\
            &  \qquad \qquad \qquad -
            M_{u_\infty}^\theta(t) \int_{[0,t]}  M_{u_\infty}^\theta(s)^{-1} \sum_{i=1}^k F_i^\theta \big(x_{u_\infty}(s) \big)v_i(s)
            \,\diff s \bigg| \\
            & \leq \big| M_{u_m}^\theta(t) - M_{u_\infty}^\theta(t) \big| \int_{[0,t]}  \big|M_{u_m}^\theta(s)^{-1}\big| \sum_{i=1}^k \big| F_i^\theta \big(x_{u_m}(s) \big) \big| \big|v_i(s) \big|
            \,\diff s \\
            & \quad + \big| M_{u_\infty}^\theta(t) \big| \int_{[0,t]}  \big|M_{u_m}^\theta(s)^{-1} - M_{u_\infty}^\theta(s)^{-1} \big| \sum_{i=1}^k \big| F_i^\theta \big(x_{u_m}(s) \big) \big| \big|v_i(s) \big|
            \,\diff s \\
            & \quad + \big| M_{u_\infty}^\theta(t) \big| \int_{[0,t]}  \big| M_{u_\infty}^\theta(s)^{-1} \big| \sum_{i=1}^k \big| F_i^\theta \big(x_{u_m}(s) \big) - F_i^\theta \big(x_{u_\infty}(s) \big) \big| \big|v_i(s) \big|
            \,\diff s
        \end{split}
    \end{equation}
for every $m\in \NN$, for every $v\in L^q([0,T],\R^k)$ and for every $t\in [0,T]$. 
Recalling that $\| w \|_{L^1([0,T])} \leq T^{1-1/q} \| w\|_{L^q([0,T])}$, since $u_m \weak_{L^q}u_\infty$, there exists $r$ such that $\| u_m \|_{L^1}\leq r$ for every $m\in \NN \cup \{\infty\}$.
Therefore, by taking the supremum with respect to $v$ and $t$ in \cref{eq:comput_y_conv} and recalling that $\|v\|_{L^1} \leq T^{1-\frac1p} \|v\|_{L^q}$, we get
\begin{equation*}
    \begin{split}
    &\sup_{\| v\|_{L^q}\leq 1} \, \sup_{t \in [0,T]} 
    \big| y_{u_m,v}^\theta(t) -  y_{u_\infty,v}^\theta(t) \big| \\
    &\leq e^{(L'+L)r} \big( C + L |x_0(\theta)| \big) T^{1-\frac1p} \bigg( \sup_{t \in [0,T]} \big| M^\theta_{u_m}(t) - M^\theta_{u_\infty}(t) \big| + 
    \sup_{t \in [0,T]} \big| M^\theta_{u_m}(t)^{-1} - M^\theta_{u_\infty}(t)^{-1} \big| \bigg) \\
    &\quad  + e^{2L'r} L T^{1-\frac1p} \sup_{t \in [0,T]} \big| x^\theta_{u_m}(t) - x^\theta_{u_\infty}(t) \big|
    \end{split}
\end{equation*}
for every $m\in \NN$, and for $\mu$-a.e.~$\theta \in \Theta$, where we used \cref{lem:bound_traj}, \cref{ass:fields}, and \cref{eq:bound_fund_mat}. 
On the one hand, using again \cref{lem:bound_traj} and \cref{eq:bound_fund_mat}, the last identity implies that 
\begin{equation} \label{eq:y_bound_Lebesgue}
    \sup_{\| v\|_{L^q}\leq 1} \, \sup_{t \in [0,T]} 
    \big| y_{u_m,v}^\theta(t) -  y_{u_\infty,v}^\theta(t) \big|
    \leq e^{(2L'+L)r}  T^{1-\frac1p} 2(2 + L) \big( C + L |x_0(\theta)| \big)
\end{equation}
for every $m\in \NN$, and for $\mu$-a.e.~$\theta \in \Theta$. 
On the other hand, by virtue of the convergences established in \cref{prop:unif_conv_map_X,lem:conv_fund_mat}, we deduce that
\begin{equation} \label{eq:y_conv_Lebesgue}
           \lim_{m\to \infty} \, 
    \sup_{\|v\|_{L^q}\leq 1}  \,
    \sup_{t\in[0,T]} \big| y_{u_m,v}^\theta(t) -  y_{u_\infty,v}^\theta(t) \big|
    =\lim_{m\to \infty} \, 
    \sup_{\|v\|_{L^q}\leq 1}  \,
    \sup_{t\in[0,T]} \big| Y^{u_m,v}_t(\theta) - Y^{u_\infty,v}_t(\theta) \big| =0
\end{equation}
for $\mu$-a.e.~$\theta \in \Theta$. Then, recalling that the right-hand side of \cref{eq:y_bound_Lebesgue} is a function in $L^{p_0}_\mu(\Theta) \subset L^{p_1}_\mu(\Theta)$, using Lebesgue convergence theorem, we obtain that 
\begin{equation*}
    \lim_{m\to \infty} \, 
    \sup_{\|v\|_{L^q}\leq 1}  \,
    \sup_{t\in[0,T]} \big\| Y^{u_m,v}_t(\cdot) - Y^{u_\infty,v}_t(\cdot) \big\|_{L^{p_1}_\mu} \leq 
        \lim_{m\to \infty} \, \big\| 
    \sup_{\|v\|_{L^q}\leq 1}  \,
    \sup_{t\in[0,T]}  Y^{u_m,v}_t(\cdot) - Y^{u_\infty,v}_t(\cdot) \big\|_{L^{p_1}_\mu} = 0,
\end{equation*}
and this concludes the proof.
\end{proof}

Building on the theory above, we are finally ready to establish the main result of this section.

\begin{theorem}[Differentiability and Compactness Properties of $\mathcal{J}_u$ and $D_u \mathcal{J}_u$] \label{thm:frechet_diff}
    Let \cref{ass:fields,ass:cauchy,ass:risk_ingred,ass:fields_C1} hold. For $\mu$-a.e.~$\theta \in \Theta$ and for every $u\in L^q([0,T], \R^k)$, let $\J_u \in L^1_\mu(\Theta,\R)$ be defined as
    \begin{equation*}
        \J_u \coloneqq \int_{[0,T]} \cdot \ \diff \nu(t) \circ A \circ X^u, \qquad \J_u(\theta) \coloneqq \int_{[0,T]} a\big(t,x_u^\theta(t),\theta \big) \,\diff \nu(t),
    \end{equation*}
    where $A\colon C^0 \big([0,T], L^{p_1}_\mu(\Theta,\R^n) \big) \to C^0 \big([0,T], L^{1}_\mu(\Theta,\R) \big)$ is the Nemytskij operator defined in \cref{eq:def_nemytskij_A} (see also \cref{rmk:restriction_Nem} for its restriction) and $\int_{[0,T]} \cdot \ \diff \nu(t)$ was introduced in \cref{eq:def_integr_op}.
    Then, $u\mapsto \J_u$ is continuously Fréchet differentiable, and for $\mu$-a.e.~$\theta \in \Theta$ and for every $v\in L^q([0,T], \R^k)$ we have:
    \begin{equation*}
        D_u\J_u[v] = \int_{[0,T]} \cdot \ \diff \nu(t) \circ A'(X^u) \big[D_u X^u[v] \big], \quad
        D_u\J_u[v](\theta) = \int_{[0,T]} D_x a\big(t,x_u^\theta(t),\theta \big) \cdot y^\theta_{u,v}(t) \, \diff \nu(t),
    \end{equation*}
    where $y^\theta_{u,v}\colon [0,T] \to \R^n$ is the trajectory described in \cref{eq:y_fund_matrix} that solves \cref{eq:ODE_y}, and where $\left( A'(Z)[\zeta] \right)_t = D_x A(Z)_t [\zeta_t]$ for every $\zeta \in C^0 \big([0,T], L^{p_1}_\mu(\Theta,\R^n) \big)$,
    with the Nemytskij operator $D_x A\colon C^0\big([0,T], L^{p_1}_\mu(\Theta,\R^n) \big) \to C^0\big([0,T], \mathcal Y)$ defined as in \cref{eq:def_nemytskij_der_A}.\\
    Finally, for every $(u_m)_m$ such that $u_m\weak_{L^q} u_\infty$ as $m \to \infty$, we have that the Fréchet differentials $\big( D_u\J_{u_m} \big)_m$ strongly converge to $D_u\J_{u_\infty}$, i.e.,
    \begin{equation*}
       \lim_{m\to \infty} \sup_{\|v\|_{L^q}\leq 1} \big\| D_u\J_{u_m}[v] - D_u\J_{u_\infty}[v]   \big\|_{L^1_\mu} =0.
    \end{equation*}
\end{theorem}
\begin{proof}
The fact that $u\mapsto \J_u$ is continuously Fréchet differentiable follows from \cref{prop:diff_nemytskij_A,prop:diff_Frechet} and from the chain rule for Fréchet differentials (see, e.g., \cite[Proposition~1.2]{ambrosetti1995primer}), and the representation of $D_u\J_u$ follows directly from the composition of the differentials.
We are left to show the strong continuity of $D_u\J_u$ when evaluated along a weakly convergent sequence of controls. To this end, we preliminary recall that if $u_m\weak_{L^p} u_\infty$ as $m\to\infty$, then
\begin{equation*}
    \lim_{m\to \infty} \left \| X^{u_m}_t -X^{u_\infty}_t \right\|_{C^0([0,T],L^{p_1}_\mu)} =0, \qquad  \lim_{m\to \infty} \, 
    \sup_{\|v\|_{L^q}\leq 1}  \big\|(D_u X^{u_m} - D_u X^{u_\infty})[v]\|_{C^0([0,T],L_\mu^{p_1})},
\end{equation*}
as established in \cref{prop:unif_conv_map_X,prop:strong_conv_diff}, respectively (recall that $p_0>p_1$ for the first convergence).
Hence, we compute
\begin{equation*}
    \begin{split}
        \sup_{\|v\|_{L^q}\leq 1} \big\| D_u\J_{u_m}[v] - D_u\J_{u_\infty}[v]   \big\|_{L^1_\mu} &= 
         \sup_{\|v\|_{L^q}\leq 1}
          \left\| A'(X^{u_m})\big[ D_uX^{u_m} [v] \big] - A'(X^{u_\infty})\big[ D_uX^{u_\infty} [v] \big]   \right\|_{L^1_\mu} \\
          & \leq \tnorm{A'(X^{u_m}) - A'(X^{u_\infty})} \sup_{\|v\|_{L^q}\leq 1} \big\|D_u X^{u_m}[v] \big\|_{C^0([0,T],L_\mu^{p_1})} \\
          & \quad + \tnorm{A'(X^{u_\infty})} \sup_{\|v\|_{L^q}\leq 1}  \big\|(D_u X^{u_m} - D_u X^{u_\infty})[v]\|_{C^0([0,T],L_\mu^{p_1})},
    \end{split}
\end{equation*}
where, for every $Z \in C^0 \big([0,T], L^{p_1}_\mu(\Theta,\R^n) \big)$, $\tnorm{A'(Z)}$ denotes the operator norm of the Fréchet differential of $A$ evaluated at $Z$. Recalling that $A$ is continuously Fréchet differentiable (see again \cref{prop:diff_nemytskij_A}), the last part of the thesis follows from the convergences and the estimate written above.
\end{proof}

\section{First-Order Optimality Conditions}\label{sec:foopt}
We are now ready to derive first-order optimality conditions for the main risk-averse optimal control problem, which takes the form:
\[
\min_{u \in \U} \mathcal{R}(\mathcal{J}(X^u)) + \alpha \rho(u)
\]
for $\alpha>0$.
The next result is a direct consequence of the discussion in \cite{ks2018existence}. We provide below a short proof. Moreover, we refer the reader to \cite{Bonalli2023} for first-order optimality conditions for risk-averse optimal control of SDEs. 
We recall that the definition of the tangent (contingent) cone to $u^{\star}$ in $\mathcal{U}$ is given by
\[
T_{\U}(u^{\star}) \coloneqq \left\{d \in L^{q}([0,T],\mathbb R^k) \left|\, \exists 
\{d_k\},\, d_k \to d,\, \exists 
\{\tau_k\} \tau_k \downarrow 0 : u^{\star}+\tau_k d_k \in \mathcal{U}\right.\right\}. 
\]
\begin{theorem}[Primal First-Order Optimality Conditions]\label{thm:pfoop}
Let \cref{ass:fields,ass:cauchy,ass:fields_C1,ass:risk_ingred,ass:running_cost,ass:sigma_1} hold. Then the following primal optimality condition holds for a minimizer $u^{\star}$.
\begin{equation}\label{eq:primal-opt}
\sup_{\vartheta \in \partial\cR(\J_{u^{\star}})} \mathbb E_{\mu}[\vartheta D_{u}\J_{u^{\star}}[\delta u]] + \alpha \rho'(u^{\star};\delta u) \ge 0 \quad \forall \delta u \in T_{\mathcal{U}}(u^{\star})
\end{equation}
\end{theorem}
\begin{proof}
This result follows from \cite[Prop.~3.13]{ks2018existence} (and the subsequent clarification in \cite{Kouri2022}), we merely verify the hypotheses. To begin, we argue that the composition $(\mathcal{R} \circ \J) \colon L^q([0,T],\R) \to \Rext$ is G\^{a}teaux directionally differentiable. This is a consequence of the chain rule, \cref{thm:frechet_diff} together with \cref{ass:sigma_1}, which imply in particular that $\cR$ is directionally differentiable in the sense of Hadamard, \cite[Prop.~2.126]{Bonnans2000}. 
\end{proof}
As noted in \cite[Thm.~2]{Kouri2018}, since $\J_{u}$ is continuously differentiable, there exists a $\vartheta^{\star} \in \partial \cR(\J_{u^{\star}})$ such that
\begin{equation}\label{eq:primal-opt-nosup}
\mathbb E_{\mu}[\vartheta^{\star} D_{u}\J_{u^{\star}}[\delta u]] + \alpha \rho'(u^{\star};\delta u) \ge 0 \quad \forall \delta u \in T_{\mathcal{U}}(u^{\star}).
\end{equation}
Although \cref{eq:primal-opt,eq:primal-opt-nosup} are convenient for characterizing all local and global solutions to the optimal control problem, they are not easy to verify or use for the development of numerical methods. To remedy this, we provide a dual, multiplier-based set of conditions here. We recall the definition of the normal cone from convex analysis given by
\[
N_{\U}(u^{\star}) := \left\{w \in L^{q'}([0,T],\mathbb R^k) \left|\, \langle w,u - u^{\star}\rangle \le 0 \quad \forall w \in \mathcal{U} \right.\right\}.
\]

\begin{theorem}[Dual First-Order Optimality Conditions]\label{thm:dfoop}
    Let \cref{ass:fields,ass:cauchy,ass:fields_C1,ass:risk_ingred,ass:running_cost,ass:sigma_1} hold and let $u^{\star} \in \U$ be an associated minimizer. Then there exist a dual multiplier $\vartheta^{\star} \in \partial\cR(\J_{u^{\star}}) \subset L^{\infty}_{\mu}(\Theta, \R)$, an adjoint state $P^{\star} \in BV\big([0,T], L^{p_1'}_{\mu}(\Theta, \R^n)\big)$, and a subgradient $\eta^{\star} \in \partial \rho(u^{\star}) \subset L^{q'}([0,T], \R^k)$ such that the following dual optimality conditions hold:
    \begin{align}
        0 &\in \lambda^{\star} + \alpha \eta^{\star} + N_{\U}(u^{\star}) \quad \text{in } L^{q'}([0,T], \R^k), \label{eq:dual_gen_eq}\\
        \begin{split} \label{eq:dual_state_eq}
        \dot X_t^{u^{\star}}(\theta) &= \sum_{i=1}^k u^{\star}_i(t) F_i^\theta\big(X_t^{u^{\star}}(\theta)\big)  \quad \text{for a.e. } t \in [0,T], \\
        X_0^{u^{\star}}(\theta) &= x_0(\theta),
        \end{split} \\
        \begin{split} \label{eq:dual_adjoint_eq}
        -\diff P_t^{\star}(\theta) &= \left[ \sum_{i=1}^k u^{\star}_i(t)\nabla_x F_i^\theta\big(X_t^{u^{\star}}(\theta)\big)  \right]^{\top} P_t^{\star}(\theta) \,\diff t + \nabla_x a\big(t, X_t^{u^{\star}}(\theta), \theta\big) \,\diff\nu(t), \\
        P_{T^+}^{\star}(\theta) &= 0,
        \end{split}
    \end{align}
    for $\mu$-a.e. $\theta \in \Theta$, where $\lambda^{\star} \in L^{q'}([0,T], \R^k)$ is defined component-wise as the expected value:
    \begin{equation}\label{eq:dual_lambda}
        \lambda^{\star}_i(t) \coloneqq \int_{\Theta} \vartheta^{\star}(\theta) P_t^{\star}(\theta)^{\top} F_i^\theta\big(X_t^{u^{\star}}(\theta)\big) \,\diff\mu(\theta) \equiv \mathbb{E}_{\theta \sim\mu}\left[ \vartheta^{\star} P_t^{\star \top} F_i^\theta(X_t^{u^{\star}}) \right], \quad i=1,\ldots,k,
    \end{equation}
    and the adjoint equation \eqref{eq:dual_adjoint_eq} is understood in the sense of measures.
\end{theorem}
% {\color{blue} Get back to classical PMP? 

% Bonnans Shapiro Lemma 6.34 provides a useful formula for normal cones to closed convex sets in $L^2$. \cite[Lemma 6.34]{Bonnans2000}}
\begin{proof}
    Let $u^{\star} \in \U$ be a minimizer. If we express \cref{thm:pfoop} using \cref{eq:primal-opt-nosup}, we may write the latter in integral form as:
    % By standard convex analysis (relying on \cref{ass:sigma_1,ass:running_cost}), there exist specific subgradients $\vartheta^{\star} \in \partial\cR(\J_{u^{\star}})$ and $\eta^{\star} \in \partial\rho(u^{\star})$ that achieve the supremum and directional derivative, yielding the variational inequality:
    \begin{equation}\label{eq:proof_vi}
        \int_{\Theta} \vartheta^{\star}(\theta) D_{u}\J_{u^{\star}}[\delta u](\theta) \,\diff\mu(\theta) + \alpha \int_{[0,T]} \eta^{\star}(t) \cdot \delta u(t) \,\diff t \ge 0 \quad \forall \delta u \in T_{\U}(u^{\star}).
    \end{equation}
    In order to derive \cref{eq:dual_gen_eq} we need to introduce the adjoint state $P$. From \cref{thm:frechet_diff}, for $\mu$-a.e.~$\theta\in \Theta$ and for every $\delta u \in T_{\U}(u^{\star})$ the Fréchet differential of the tracking term evaluated in the direction $\delta u$ is given by:
    \begin{equation*}
        D_u\J_{u^{\star}}[\delta u](\theta) = \int_{0}^T \nabla_x a\big(t,X_t^{u^{\star}}(\theta),\theta \big) \cdot y^\theta(t) \, \diff \nu(t),
    \end{equation*}
    where $y^\theta \coloneqq y^\theta_{u^{\star},\delta u} \in AC([0,T],\R^n)$ solves \cref{eq:ODE_y}. \\
    To isolate $\delta u$, we introduce the $\theta$-dependent adjoint state $P^{\star} \in BV\big([0,T], L^{p_1'}_{\mu}(\Theta, \R^n)\big)$ satisfying the backward linear measure differential equation \eqref{eq:dual_adjoint_eq} (see, e.g., \cite{Imaz}). For $\mu$-a.e.~$\theta\in \Theta$, because $y^\theta$ is absolutely continuous and $P^{\star}(\theta)$ is of bounded variation, we apply the integration by parts formula for Lebesgue-Stieltjes integrals \cite[Theorem~6.2.2~(i)]{Carter2000}:
    \begin{equation} \label{eq:P_star_y}
        \begin{split}
        \diff\big( P^{\star}(\theta) \cdot y^\theta \big)([0,T]) &= \big(P^{\star}(\theta) \cdot y^\theta\big)|_{T^+} - \big(P^{\star}(\theta) \cdot y^\theta\big)|_{0^-} \\
        &= \int_{[0,T]} P_t^{\star}(\theta) \cdot \dot{y}^\theta(t) \,\diff t + \int_{[0,T]} y^\theta(t) \cdot \, \diff P_t^{\star}(\theta).
        \end{split}
    \end{equation}
    Because $P_{T^+}^{\star}(\theta) = 0$ and $y^\theta(0) = 0$, the first line in the previous relation is zero. Substituting the dynamics of $\dot{y}^\theta(t)$ from \eqref{eq:ODE_y} and $\diff P_t^{\star}(\theta)$ from \eqref{eq:dual_adjoint_eq} into the right-hand side of \eqref{eq:P_star_y} yields:
    \begin{equation*}
        \begin{split}
            0 &= \int_{[0,T]} P_t^{\star}(\theta) \cdot \left[ \sum_{i=1}^k \nabla_x F_i^\theta\big(X_t^{u^{\star}}(\theta)\big) u^{\star}_i(t) y^\theta(t) + \sum_{i=1}^k F_i^\theta\big(X_t^{u^{\star}}(\theta)\big) \delta u_i(t) \right] \,\diff t \\
            &\quad - \int_{[0,T]} y^\theta(t) \cdot \left[ \left(\sum_{i=1}^k \nabla_x F_i^\theta\big(X_t^{u^{\star}}(\theta)\big) u^{\star}_i(t)\right)^{\top} P_t^{\star}(\theta) \,\diff t + \nabla_x a\big(t, X_t^{u^{\star}}(\theta), \theta\big) \,\diff\nu(t) \right].
        \end{split}
    \end{equation*}
    Canceling like terms and rearranging what remains, we obtain the identity:
    \begin{equation*}
        \int_{[0,T]} \nabla_x a\big(t, X_t^{u^{\star}}(\theta), \theta\big) \cdot y^\theta(t) \,\diff\nu(t) = \int_{[0,T]} P_t^{\star}(\theta)^{\top} \left( \sum_{i=1}^k F_i^\theta\big(X_t^{u^{\star}}(\theta)\big) \delta u_i(t) \right) \,\diff t
    \end{equation*}
    for $\mu$-a.e.~$\theta\in \Theta$; we recognize the left-hand side as $D_{u}\J_{u^{\star}}[\delta u](\theta)$.
    % By virtue of \cref{thm:frechet_diff}, we recognize the left-hand side as $D_{u}\J_{u^{\star}}[\delta u](\theta)$. 
    Therefore, the variational inequality \eqref{eq:proof_vi} takes the form:
    \begin{equation*}
        \int_{\Theta} \vartheta^{\star}(\theta) \left( \int_{[0,T]} P_t^{\star}(\theta)^{\top} \sum_{i=1}^k F_i^\theta\big(X_t^{u^{\star}}(\theta)\big) \delta u_i(t) \,\diff t \right) \diff\mu(\theta) + \alpha \int_{[0,T]} \eta^{\star}(t) \cdot \delta u(t) \,\diff t \ge 0 
    \end{equation*}
    for every $\delta u \in T_{\U}(u^{\star})$.
    Since the vector fields are bounded (\cref{ass:fields}), the adjoint state is integrable, and $\mu$ is a probability measure, the integrands are absolutely integrable. By Fubini's theorem, we obtain:
    \begin{equation*}
        \int_{[0,T]} \left( \int_{\Theta} \vartheta^{\star}(\theta) P_t^{\star}(\theta)^{\top} F^\theta\big(X_t^{u^{\star}}(\theta)\big) \,\diff\mu(\theta) \right) \cdot \delta u(t) \,\diff t = \int_{[0,T]} \lambda^{\star}(t) \cdot \delta u(t) \,\diff t,
    \end{equation*}
    where we set $\lambda^{\star}$ as in \eqref{eq:dual_lambda}. Inserting this back into \eqref{eq:proof_vi} yields:
    \begin{equation*}
        \int_{[0,T]} \big(\lambda^{\star}(t) + \alpha \eta^{\star}(t)\big) \cdot \delta u(t) \,\diff t \ge 0 \quad \forall \delta u \in T_{\U}(u^{\star}).
    \end{equation*}
    By the definition of the normal cone $N_{\U}(u^{\star})$, this is exactly $-(\lambda^{\star} + \alpha \eta^{\star}) \in N_{\U}(u^{\star})$, which proves the generalized equation \eqref{eq:dual_gen_eq}. 
    % \textcolor{blue}{Is \cref{eq:dual_gen_eq} satisfied at a.e. $t$? Is this apparent from this proof?}
\end{proof}

\begin{remark}[Connection to the Classical PMP] \label{rmk:pmp_connection}
The abstract result in \cref{thm:dfoop} can be compared with the classical Pontryagin Maximum/Minimum Principle (PMP) in a natural way. To this end, we first make some simplifications on the problem data. Assume that $\nu$ is the standard Lebesgue measure (i.e., $\diff\nu(t) = \diff t$) and that $\U \coloneqq \{ u \in L^q([0,T], \R^k) \mid u_{\min} \leq u_i(t) \leq u_{\max}, \ i=1,\ldots,k, \ \text{ for a.e. } t\}$.

For each $\theta \in \Theta$, we define the parameter-dependent Hamiltonian $H^\theta \colon [0,T] \times \R^n \times \R^k \times \R^n \to \R$ as
\begin{equation*}
    H^\theta(t, x, u, p) \coloneqq p^\top \left( \sum_{i=1}^k u_i F_i^\theta(x) \right) + a(t, x, \theta).
\end{equation*}
In addition, we define an ensemble Hamiltonian $\mathcal{H}$ evaluated along the ensemble trajectories:
\begin{equation*}
    \mathcal{H}\big(t, X_t(\cdot), u, P_t(\cdot), \vartheta^\star\big) \coloneqq \mathbb{E}_{\theta \sim \mu} \Big[ \vartheta^\star(\theta) H^\theta\big(t, X_t(\theta), u, P_t(\theta)\big) \Big] + \alpha f(t, u).
\end{equation*}
For $\mu$-a.e.~$\theta \in \Theta$, the state and adjoint equations (\cref{eq:dual_state_eq,eq:dual_adjoint_eq}) are equivalent to the Hamiltonian system:
\begin{equation*}
    \dot{X}_t^{\star}(\theta) = \nabla_p H^\theta\big(t, X_t^{\star}(\theta), u^{\star}(t), P_t^{\star}(\theta)\big), \qquad
    -\dot{P}_t^{\star}(\theta) = \nabla_x H^\theta\big(t, X_t^{\star}(\theta), u^{\star}(t), P_t^{\star}(\theta)\big).
\end{equation*}
Finally, the generalized equation \eqref{eq:dual_gen_eq} acts as the Minimum Principle. At a.e.~$t \in [0,T]$, $u^{\star}(t)$ must minimize ${\mathcal{H}}$ over the box constraints. For instance, if $f(t,u) = \frac{1}{2}|u|^2$, this yields the explicit projection:
\begin{equation*}
    u_i^{\star}(t) = \max \left( u_{\min}, \min \left( u_{\max}, -\frac{1}{\alpha} \mathbb{E}_{\theta \sim \mu} \Big[ \vartheta^{\star}(\theta) P_t^{\star}(\theta)^\top F_i^\theta\big(X_t^{\star}(\theta)\big) \Big] \right) \right), \quad i=1,\ldots,k.
\end{equation*}
Here, we see that the multiplier $\vartheta^\star$ effectively re-weights the classical optimality conditions to prioritize the important risk scenarios identified by the risk measure $\cR$.
\end{remark}

\begin{remark}[Effects of Risk Measures on the Optimal Control] \label{rem:cvar-effect}
The reader unfamiliar with the literature on optimization with risk measures, which is traditionally a topic of operations research, may wonder what exactly all this formalism provides. This is perhaps best understood with a simple example. 

Suppose $\mathcal{U}$ constitutes the entire control space, and $\rho(u) = \frac{1}{2} \| u \|^2_{\mathcal{U}}$. %, and each $F^{\theta}$ only depends on $\theta$.  
In addition, assume that $\mu$ is replaced by the empirical probability measure for a random sample of size $S$. 
%If we first replace $F^{\theta}$ by its average $\bar{F}$, then \eqref{eq:dual_gen_eq} states that
%\[
%u^{\star}_i(t) = -\frac{1}{\alpha} (P^{\star}_t)^{\top} \bar{F}_i.
%\]
A na\"{i}ve approach would consist in considering the empirical average $\bar \theta \coloneqq \frac1S \sum_{s=1}^S \theta_s$, and solving the optimal control problem for the specific value $\theta=\bar\theta$.
The corresponding control action, obtained by simply replacing the uncertain parameters by their means, is the furthest from being a robust. In contrast, the risk neutral case, which amounts to using $\mathcal{R} = \mathbb E$ and setting $\vartheta^{\star} \equiv 1$, yields an optimal control that averages over all scenarios 
\[
u^{\star}_i(t) = -\frac{1}{\alpha S} \sum_{s = 1}^{S}\left[(P^{\star}_t(\theta_s))^{\top} F^{\theta_s}_i\right].
\]
While most likely more robust than the previous solution, it still weights all scenarios equally: we are \textit{neutral} to both good and bad outlier scenarios. 

Finally, we compare this to a risk-averse scenario in which we use the so-called \textit{average value-at-risk level $\beta \in (0,1)$}. In the continuous setting, letting $F^{-1}_{X}$ denote the quantile function for a random variable $X$, this is defined by the formulae:
\begin{equation} \label{eq:def_AVaR}
    \mathcal{R}(X) = \mathrm{AVaR}_{\beta}(X)= \frac{1}{1-\beta}\int_{\beta}^{1} F^{-1}_{X}(t) dt \coloneqq \inf_{t \in \mathbb R}\left\{t + \frac{1}{1-\beta} \mathbb E[\max\{0,X-t\}]\right\}.
\end{equation} 
However, when dealing with a sample of $S$ independent and identically distributed (i.e., i.i.d.) observations sorted in non-decreasing order $X_{(1)} \le X_{(2)} \le \dots \le X_{(S)}$, $\mathrm{AVaR}_{\beta}(X)$ takes the form
\[
\mathrm{AVaR}_{\beta}(X) = \frac{1}{S(1-\beta)} \left[ \sum_{s=\lfloor S\beta \rfloor + 1}^{S} X_{(s)} + \left( \lfloor S\beta \rfloor + 1 - S\beta \right) X_{(\lfloor S\beta \rfloor)} \right].
\]
Notice that the focus is on the worst $(1-\beta) \times 100$ percent outcomes in the tail. Now, the optimal $\vartheta^{\star}$, known as the \textit{risk identifier}, weights the $S$ samples automatically. These weights can be obtained by solving the simple linear program (fixing the optimal $u^{\star}$): 
\[
\max_{\vartheta \in \mathbb R^{S}}\left\{\frac{1}{S} \sum_{s=1}^S \vartheta_s \mathcal{J}(u^{\star}; \theta_s) : 0 \le \vartheta_s \le \frac{1}{1-\beta}, \quad \frac{1}{S} \sum_{s=1}^{S} \vartheta_{s} = 1 \right\}.
\]
From the structure of this optimization, it follows that the support of $\vartheta^{\star}$ consists of exactly $K = \lceil S(1-\beta) \rceil$ elements corresponding to the largest cost realizations. Letting $(s)$ denote the index of the $s$-th largest realization, the optimal control reduces to a weighted average over only these tail scenarios:
\[
u^{\star}_i(t) = -\frac{1}{\alpha S} \sum_{k=S-K+1}^{S} \vartheta^{\star}_{(k)} \left[(P^{\star}_t(\theta_{(k)}))^{\top} F^{\theta_{(k)}}_i\right].
\]
By focusing on the highest impact scenarios, this control action is more robust than the mean-value or risk-neutral settings. Nevertheless, the remaining scenarios are not ignored; minimizing the AVaR ensures that the more favorable realizations remain well within the system's safe operating limits. This follows from the fact that if $t^{\star} \in \mathbb R$ is an optimal solution to the variational problem
\[
\mathcal{R}(\mathcal{J}^{u^\star}) = t^{\star} + \frac{1}{(1-\beta)S}\sum_{s = 1}^{S}\max\{0,\mathcal{J}(u^{\star};\theta_s) - t^{\star}\},
\]
then $t^{\star}$ is effectively the upper $\beta$-quantile, which ensures that $\mathcal{J}(u^{\star};\theta) \le t^{\star}$ with probability at least $\beta$.
\end{remark}

\section{Numerical experiment on a quantum system} \label{sec:num}
We propose to solve the optimal control problem using the primal dual minimization algorithm introduced in \cite{Kouri2021}. It is here that we require the weak-to-strong continuity properties of the control-to-state mapping and its derivative, as the latter two properties are necessary to guarantee convergence in the continuous setting, see \cite[Theorem 3]{Kouri2021}.

In the numerical experiment\footnote{The codes for reproducing the experiment is available at the following link: \url{https://doi.org/10.5281/zenodo.20020902}}, we compare the risk-averse approach with the resolution of the averaged and the minimax ensemble optimal control problems, in the same quantum control framework studied in \cite[Section~6]{Scag_25}.
More specifically, we study a two-level system governed by the Schr\"odinger equation
\begin{equation} \label{eq:qubit_sys}
i\dot\psi_u^\theta(t) = \big(H_0^\theta + H_1 u(t) \big)\psi_u^\theta(t) \quad \mbox{a.e.~in } [0,T],
 \qquad \psi_u^\theta(0) =
\left(
\begin{matrix}
0\\
1
\end{matrix}
\right),
\end{equation}
where
\begin{equation*}
H_0^\theta \coloneqq \left(
\begin{matrix}
E+\theta & 0 \\
0 & -E-\theta
\end{matrix}
\right),
\quad
H_1 \coloneqq  \left(
\begin{matrix}
0 & 1 \\
1 & 0
\end{matrix}
\right).
\end{equation*}
Here, $\psi_u^\theta\colon [0,T]\to\mathbb{C}^2$ denotes the state, $E>0$, and $\theta\in[\theta_0,\theta_1]$ represents the unknown parameter indexing the ensemble. The function $u\colon [0,T]\to\R$ is a real-valued control input.
In \cite{RABS}, the authors addressed the problem of constructing a control such that, at the final time $T$, the trajectories of the ensemble \eqref{eq:qubit_sys} approach the target state $\psi^{\mathrm{tar}}\coloneqq (1,0)^\top$ (up to a phase that may depend on $\theta$). Equivalently, the objective was to maximize $\inf_{\theta\in [\theta_0,\theta_1]} \big|\langle \psi^{\mathrm{tar}} \mid \psi^{\theta}_u(T) \rangle\big|$. Their main result established uniform controllability for \eqref{eq:qubit_sys} (see \cite[Theorem~3]{RABS}), along with an explicit construction of control families achieving this objective (see \cite[Remark~5]{RABS}).

\begin{figure}[htbp]
\centering

\includegraphics[width=0.48\textwidth]{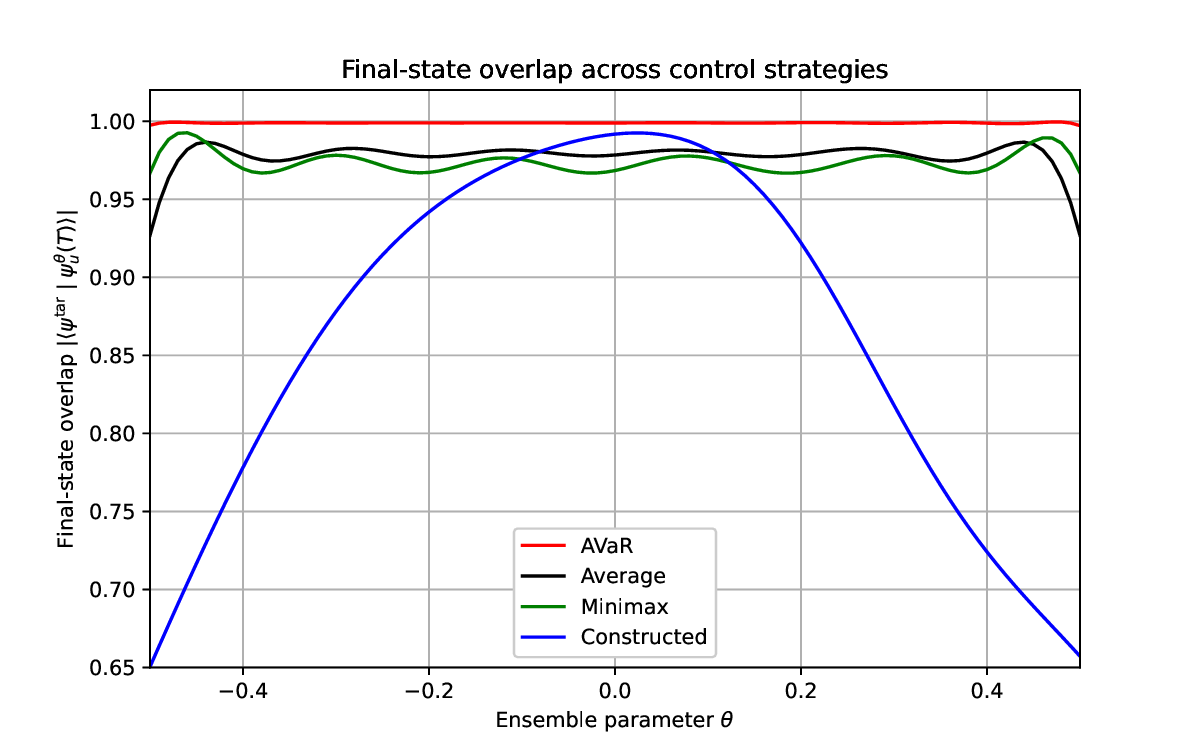}
\hfill
\includegraphics[width=0.48\textwidth]{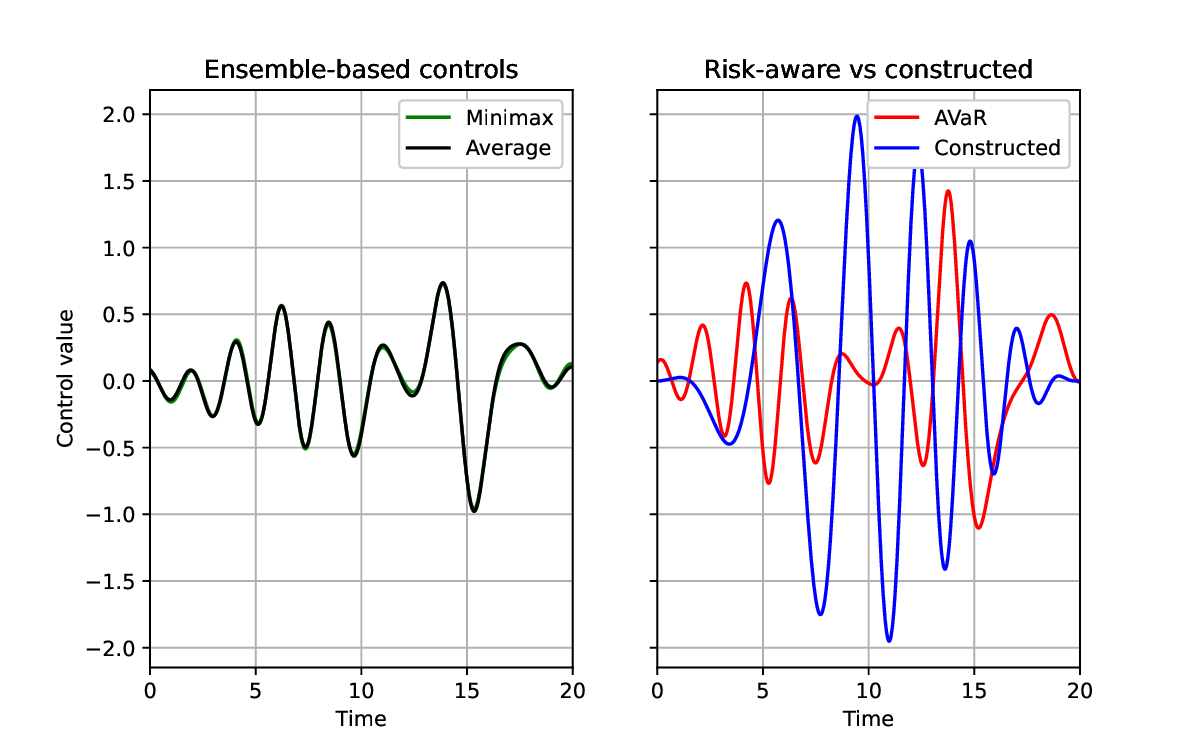}

\caption{Comparison of control strategies. 
Left: final-state overlap $|\langle \psi^{\mathrm{tar}} \mid \psi_u^\theta(T)\rangle|$ as a function of the parameter $\theta \in \Theta$, for controls obtained by minimizing $\I_{\mathrm{avg}}$, $\I_{\mathrm{worst}}$, and $\I_{\mathrm{AVaR}}$, together with the constructed control from \cite{RABS}. 
Right: corresponding control inputs $u^\star_{\mathrm{avg}}$, $u^\star_{\mathrm{worst}}$, and $u^\star_{\mathrm{AVaR}}$ over the time interval $[0,T]$. The AVaR-based control achieves the best uniform performance across the ensemble.}
\label{fig:results}
\end{figure}

We consider $\Theta=[-0.5, 0.5]$ and $\Theta_{N}\coloneqq \{ -0.5+\frac{n}{N} : n=0,\ldots,N \}$, we denoted with $\mu_N$ the uniform probability measure over $\Theta_N$, and we set $N=100$. For a rigorous discussion on approximations based on discrete or empirical measures, we refer to \cite{Milz2024}.
We control the system in \cref{eq:qubit_sys} in the evolution interval $[0,T]$ with $T=20$, and we compute the control inputs by minimizing the functionals $\I_{\mathrm{avg}},\I_{\mathrm{worst}},\I_{\mathrm{AVaR}} \colon L^2([0,T],\R)\to \R$ defined, respectively, as follows:
\begin{equation} \label{eq:def_qubit_worst}
\I_{\mathrm{worst}}(u)\coloneqq  \max_{\theta \in \Theta_N} \left(
1- \big|\langle \psi^{\mathrm{tar}} | \psi^{\theta}_u(T) \rangle\big|^2
\right) + \alpha \| u \|_{L^2}^2,
\end{equation}
\begin{equation} \label{eq:def_qubit_avg}
\I_{\mathrm{avg}}(u)\coloneqq  \mathbb E_{\theta\sim \mu_N } \Big[  
1- \big|\langle \psi^{\mathrm{tar}} | \psi^{\theta}_u(T) \rangle\big|^2 \Big]
  + \alpha \| u \|_{L^2}^2,
\end{equation}
and
\begin{equation} \label{eq:def_qubit_AVaR}
\I_{\mathrm{AVaR}}(u)\coloneqq   \mathrm{AVaR}_{\beta } \Big[  
1- \big|\langle \psi^{\mathrm{tar}} | \psi^{\theta}_u(T) \rangle\big|^2 \Big]
  + \alpha \| u \|_{L^2}^2,
\end{equation}
where $\mathrm{AVaR}_{\beta }$ is defined as in \cref{eq:def_AVaR}, and where we set $\alpha = 2^{-4}$ in \cref{eq:def_qubit_avg,eq:def_qubit_worst,eq:def_qubit_AVaR} and the risk-level $\beta = 0.95$ in \cref{eq:def_qubit_AVaR}.
In the numerical implementation, we adopt a uniform time-step discretization with $\Delta t = 2^{-5}$, and we consider controls $u$ that are piecewise constant on each subinterval $[k\Delta t, (k+1)\Delta t)$, for $k = 0,1,\ldots$. On every subinterval, the system \eqref{eq:qubit_sys} is propagated by computing the \emph{exact matrix exponential} associated with $H_0^\theta + H_1 u$.

The numerical experiments are initialized with a Gaussian random control, generated in a reproducible manner. All derivatives involved in the optimization procedures are computed using automatic differentiation (PyTorch).
We first minimize the functional $\I_{\mathrm{avg}}$ by performing a (deterministic) gradient descent for $500$ iterations, with an adaptive adjustment of the step size based on the Armijo rule. The resulting control is denoted by $u^\star_{\mathrm{avg}}$, and it is then used as an initial guess for the minimization of $\I_{\mathrm{worst}}$ and $\I_{\mathrm{AVaR}}$. For $\I_{\mathrm{worst}}$, we employ a subgradient method with a learning rate that decays proportionally to the inverse square root of the iteration number, starting from the an initial step size equal to $2^{-3}$. The procedure is run for $1000$ iterations.
For the minimization of $\I_{\mathrm{AVaR}}$, we implement \cite[Algorithm~2]{Kouri2021}, with a PyThorch built-in L-BFGS scheme for the resolution of the inner optimization sub-problem. The stopping criterion of the primal-dual algorithm is satisfied after $15$ iterations.
The results are reported in \cref{fig:results}. The control obtained by minimizing $\I_{\mathrm{AVaR}}$ exhibits the best performance among the considered approaches. For comparison, we also include the control constructed in \cite{RABS}, which is defined by
\begin{equation*}
u_{\mathrm{ref}}(t)
=
2\varepsilon_1 \big(1 - \cos(2\pi \varepsilon_1 \varepsilon_2 t)\big)
\cos\left(
2Et
+ \frac{v_0 - v_1}{\pi \varepsilon_1 \varepsilon_2}\sin(\pi \varepsilon_1 \varepsilon_2 t)
+ (v_0 + v_1)t
\right),
\end{equation*}
where $\varepsilon_1 = 0.5$, $\varepsilon_2 = 0.1$, $v_0=-0.5$ and $v_1=0.5$.

\section*{Conclusions}
% In this paper, we investigated risk-averse ensemble optimal control problems for nonlinear-in-state control-affine systems within a general framework, encompassing both the assumptions on the controlled fields and the structure of the cost functional. 
% Particular attention was devoted to the regularity properties of the control-to-state map, valued in $C^0 \big( [0,T],L^{p_0}_\mu(\Theta) \big)$. We established its Fréchet differentiability and the weak-to-strong continuity of its differential with respect to the base point. 
% These results enabled the application of an efficient primal-dual algorithm for solving risk-averse ensemble optimal control problems, which was validated on a quantum control example.

This work demonstrates that transitioning from risk-neutral ensemble optimal control to a risk-averse paradigm is essential to properly hedge against worst-case parametric outliers in a variety of relevant modern applications. 
For example, we observed how specific risk measures such as the Average Value-at-Risk can be used to modulate between the risk-neutral and robust settings. 
On the one hand, the value of this approach from a practical standpoint was particularly evident in the numerical experiment motivated by quantum control. 
On the other hand, our theoretical results for control-affine systems establishing the continuous Fr\'echet differentiability, as well as the weak-to-strong continuity of both the control-to-state mapping and its differential, provide the analytical prerequisites needed to deploy infinite-dimensional primal-dual optimization algorithms such as the one developed in \cite{Kouri2021}. Furthermore, these regularity results lay the essential groundwork for asymptotic statistical studies of empirical approximations, as in \cite{Milz2024}. Future work will focus on exploring applications to the robust training of Neural ODEs, as well as extending this framework to fully nonlinear systems; a challenging direction that will likely require analytical relaxation of the control space via Young measures.

\subsection*{Acknowledgements}

Alessandro Scagliotti acknowledges support from the ERC Advanced Grant NEITALG, grant agreement No.~101198055 (P.I.: Prof.~Massimo Fornasier). Thomas M. Surowiec was supported by the Research Council of Norway (grant number: 357482, SURE-AI Centre for Sustainable, Risk-Averse and Ethical AI).\\
%The funding for Open Access publishing have been provided by the Technical University of Munich within the German DEAL agreement.

\bigskip
\begin{center}
  \FundingLogos
  
  \vspace{0.5em}
  \begin{tcolorbox}\centering\small
    % CHOOSE the sentence that applies to your grant/programme and replace placeholders.
    % 1) For Horizon Europe / ERC (recommended wording):
    Funded by the European Union. Views and opinions expressed are however those of the author(s) only and do not necessarily reflect those of the European Union or the European Research Council Executive Agency. Neither the European Union nor the granting authority can be held responsible for them.
    This project has received funding from the European Research Council (ERC) under the European Union’s Horizon Europe research and innovation programme (grant agreement No.~101198055, project acronym NEITALG).
    
    % 2) (Alternatively, for Horizon 2020 ERC grants use the H2020 wording:)
    % This project has received funding from the European Research Council (ERC) under the European Union's Horizon 2020 research and innovation programme (grant agreement No. <GRANT NUMBER>, project acronym <ACRONYM>).
  \end{tcolorbox}
\end{center}

%%%%%%%%%%%%%%%%%%%%
%%%%%%%%%%%%%%%%%%%%%

\appendix
\section{Technical results} \label{app:technical}

We recall below the basic results that guarantee existence of solutions to differential equations in Banach spaces.
We first recall the notion of $L^q$-Carath\'eodory function. For further details, we refer the reader to \cite[Chapters~3 and~16]{o1997existence}.

\begin{definition} \label{def:carath}
    Let $(E, \|\cdot\|_E)$ be a Banach space. A function $F\colon [0,T] \times E \to E$ is said to be \emph{$L^q$-Carath\'eodory} if the following conditions hold:
    \begin{itemize}
        \item The map $y \mapsto F(t, y)$ is continuous for a.e.~$t\in [0,T]$;
        \item The map $t \mapsto F(t, y)$ is measurable for every $y\in E$;
        \item For every $c > 0$ there exists $h_c\in L^q([0,T],\R_+)$ such that $\|y\|_E\leq C$ implies $\| F(t,y) \|_E \leq h_c(t)$ for a.e.~$t\in [0,T]$.
    \end{itemize}
    Finally, we say that a {$L^q$-Carath\'eodory} function $F$ is \emph{$L^q$-Lipschitz} if there exists $\ell \in L^q([0,T],\R_+)$ such that 
    \begin{equation*}
        \| F(t,y) - F(t,y') \|_E \leq \ell(t) \| y - y' \|_E
    \end{equation*}
    for a.e.~$t\in [0,T]$.
\end{definition}

We now state the existence result.

\begin{theorem} \label{thm:existence_ODE}
    Let $(E, \|\cdot\|_E)$ be a Banach space, and let $F\colon [0,T] \times E \to E$ be $L^q$-Carath\'eodory and $L^q$-Lipschitz. Then, for every $y_0\in E$, the Cauchy problem 
    \begin{equation*}
        \begin{cases}
            \dot y = F(t,y)  &
\mbox{for a.e. }t\in [0,T],\\
            y(0) =y_0,
        \end{cases}
    \end{equation*}
    admits a unique solution $y(\cdot) \in W^{1,p}([0,T],E)$.
\end{theorem}
\begin{proof}
    See \cite[Theorem~3.4]{o1997existence} and the subsequent Remark. See also \cite[Theorem~16.2]{o1997existence} for a statement in full generality.
\end{proof}

\begin{lemma} \label{lem:lipsch_traj}
Let \cref{ass:fields,ass:cauchy} hold. Then, for $\mu$-a.e.~$\theta \in \Theta$ and for every $u,v \in L^q([0,T],\R^k)$ with $\|v\|_{L^q}\leq 1$ 
we have 
\begin{equation*} %\label{eq:lip_x}
\begin{split}
\sup_{t\in [0,T]} \big| x^\theta_{u+v}(t) - x_{u}^\theta(t) \big| &\leq    
\left( C + L   
            \Big( | x_0(\theta) | + C \| u \|_{L^1}  \Big)e^{L \| u \|_{L^1}}
            \right) 
            e^{L \| u \|_{L^1}} \|v\|_{L^1}\\
            &\leq \left( C + L   
            \Big( | x_0(\theta) | + C \| u \|_{L^1}  \Big)
            \right) 
            e^{2L \| u \|_{L^1}} \|v\|_{L^1}.
\end{split}
\end{equation*}
\end{lemma}
\begin{proof}
    Using the fact that $x_{u+v}^\theta(\cdot), x_{u}^\theta(\cdot)$ solve \eqref{eq:ens_ctrl_sys}, for every $t\in [0,T]$ we deduce that
    \begin{equation*}
        \begin{split}
            \big| x^\theta_{u+v}(t) - x_{u}^\theta(t) \big| & \leq
            \sum_{i=1}^k \int_{[0,t]} 
             \left| F_i^\theta \big( x_{u+v}^\theta(s) \big) \right| |v_i(s)|
            + \left| F_i^\theta \big( x_{u+v}^\theta(s) \big) - F_i^\theta \big( x_{u}^\theta(s) \big) \right| |u_i(s)|
            \, \diff s\\
            &\leq \sum_{i=1}^k \int_{[0,t]}
            \left( C + L\,  
            {\textstyle \sup_{s \in[0,t]}} \big| x_{u}^\theta(s) \big|
            \right)
             |v_i(s)|
             + L |u_i(s)| \big| x^\theta_{u+v}(s) - x_{u}^\theta(s) \big|
            \, \diff s,
        \end{split}
    \end{equation*}
    where we used the sublinear growth condition implied by \cref{ass:fields}, together with the Lipschitz-continuity of the vector fields. 
    Finally, the statement follows from the Gr\"onwall Lemma and from the estimate in \cref{lem:bound_traj}.
\end{proof}

\subsection{Properties of the Nemytskij operators} \label{subsec:Nemytskij}

\begin{lemma} \label{lem:lipschitz_nemytskij_A}
    Let \cref{ass:risk_ingred} hold. Let $A\colon L^\infty\big([0,T], L^{p_1}_\mu(\Theta,\R^n) \big) \to L^\infty\big([0,T], L^{1}_\mu(\Theta,\R) \big)$ be the Nemytskij operator defined in \cref{eq:def_nemytskij_A}.
    Then, $A$ is locally Lipschitz continuous, i.e., for every $R>0$ there exists a constant $L_A(R)>0$ (depending on $R$) such that 
    \begin{equation*}
        \sup_{t\in [0,T]} \| A(Z)_t - A(Z')_t \|_{L^1_\mu(\Theta,\R)} \leq L_A(R) \sup_{t\in [0,T]} \| Z_t - Z'_t \|_{L^{p_1}_\mu(\Theta,\R^n)},
    \end{equation*}
    whenever $\max \Big\{ \sup_{t\in [0,T]} \| Z_t \|_{L^{p_1}_\mu(\Theta,\R^n)}, \sup_{t\in [0,T]} \| Z_t' \|_{L^{p_1}_\mu(\Theta,\R^n)} \Big\} \leq R$.
\end{lemma}
\begin{proof}
    Let us consider $Z,Z' \in L^\infty\big([0,T], L^{p_1}_\mu(\Theta,\R^n) \big)$. Then, using \cref{ass:risk_ingred}, we compute
    \begin{equation*}
        \begin{split}
            \sup_{t\in [0,T]} \big\| A(Z)_t & - A(Z')_t \big\|_{L^1_\mu(\Theta,\R)}
             = \sup_{t\in [0,T]} \int_\Theta
            \left| a\big(t,Z_t(\theta),\theta\big) - a\big(t,Z_t'(\theta),\theta\big) \right|
            \, \diff\mu(\theta)\\
            & \leq \sup_{t\in [0,T]}
            \int_\Theta
            \Big( L\big(|Z_t(\theta)|,\theta \big) + L\big( |Z_t'(\theta)|, \theta \big) \Big) |Z_t(\theta) - Z_t'(\theta)|
            \, \diff\mu(\theta) \\
            &\leq 
            \sup_{t\in [0,T]}\left( \left\|
            L\big( Z_t(\cdot), \cdot \big) + L\big( Z_t'(\cdot), \cdot \big) \right\|_{L^{p_1'}_\mu(\Theta,\R)} \|Z_t -Z_t'\|_{L^{p_1}_\mu(\Theta,\R^n)}
            \right).
        \end{split}
    \end{equation*}
Moreover, recalling that $(r,\theta) \mapsto L(r,\theta) = C^{\mathrm{Lip}} r^{p_1-1} + g^{\mathrm{Lip}}(\theta)$, we obtain that
\begin{equation*}
    \begin{split}
    \sup_{t\in [0,T]} \big\|
            L\big( Z_t(\cdot), \cdot \big) &+ L\big( Z_t'(\cdot), \cdot \big) \big\|_{L^{p_1'}_\mu(\Theta,\R)} \\
            &
            \leq 2 \| g^{\mathrm{Lip}} \|_{L^{p_1'}_\mu(\Theta,\R)} + \sup_{t\in [0,T]} \left(
            \left\| |Z_t|^{p_1-1} + |Z_t'|^{p_1-1} \right\|_{L^{p_1'}_\mu(\Theta,\R)}
            \right)\\
            & \leq
            2 \| g^{\mathrm{Lip}} \|_{L^{p_1'}_\mu(\Theta,\R)} + \sup_{t\in [0,T]} \left(
            \| Z_t\|_{L^{p_1}_\mu(\Theta,\R^n)}^{p_1-1} 
            + \| Z_t'\|_{L^{p_1}_\mu(\Theta,\R^n)}^{p_1-1} 
            \right), 
    \end{split}
\end{equation*}
and this concludes the proof.
\end{proof}

\begin{lemma} \label{lem:cont2cont_nemytskij_A}
    Let \cref{ass:risk_ingred} hold. Let $A\colon L^\infty\big([0,T], L^{p_1}_\mu(\Theta,\R^n) \big) \to L^\infty\big([0,T], L^{1}_\mu(\Theta,\R) \big)$ be the Nemytskij operator defined in \cref{eq:def_nemytskij_A}. 
    If $Z\in C^0\big([0,T], L^{p_1}_\mu(\Theta,\R^n) \big)$, then $A(Z)\in C^0\big([0,T], L^{1}_\mu(\Theta,\R) \big)$.
\end{lemma}

\begin{proof}
    Let us consider $t\in [0,T]$ and any sequence $(t_m)_m$ such that $t_m \to t$ as $m\to\infty$. As $Z\in C^0\big([0,T], L^{p_1}_\mu(\Theta,\R^n) \big)$, it turns out that $Z_{t_m}\to_{L^{p_1}_\mu} Z_t$ as $m\to \infty$. We shall argue by using the Urysohn Lemma: We consider the sequence $\big(A(Z)_{t_m}\big)_m$, and we show that every subsequence $\big(A(Z)_{t_{m_k}}\big)_k$ has a (not relabeled) sub-subsequence that converges to $A(Z)_t$ in the $L^1_\mu$-norm.
    To see this, we fix a subsequence $\big(A(Z)_{t_{m_k}}\big)_k$, and we take a sub-subsequence such that $Z_{t_{m_k}}(\theta) \to Z_t(\theta)$ for $\mu$-a.e.~$\theta$ and such that $|Z_{t_{m_k}}|\leq h$, with $h\in L^{p_1}_\mu(\Theta,\R)$.
    Moreover, recalling the definition of $A$ in \cref{eq:def_nemytskij_A}, by virtue of the continuity of $(t,x)\mapsto a(t,x,\theta)$ (see \cref{ass:risk_ingred}), we deduce that, for $\mu$-a.e.~$\theta\in \Theta$,
    \begin{equation*}
        A(Z)_{t_{m_k}}(\theta) = a\left(t_{m_k}, Z_{t_{m_k}}(\theta), \theta \right) \longrightarrow a(t, Z_{t}(\theta), \theta ) = A(Z)_t(\theta) \quad \mbox{as } k\to \infty.
    \end{equation*}
    Moreover, leveraging on \cref{eq:bound_a}, we have that
    \begin{equation}
        \left| A(Z)_{t_{m_k}}(\theta) \right| \leq C^{\mathrm{Lip}}|h(\theta)|^{p_1} + 2 g^{\mathrm{Lip}}(\theta)|h(\theta)|+ \bar g_a(\theta)
    \end{equation}
    for $\mu$-a.e.~$\theta\in \Theta$ and for every $k\geq1$. Recalling that $g^{\mathrm{Lip}} \in L^{p_1'}_\mu(\Theta,\R)$ and $\bar g_a \in L^{1}_\mu(\Theta,\R)$, we conclude that $A(Z)_{t_{m_k}} \to_{L^1_\mu}  A(Z)_t$ as $k\to\infty$ by the Lebesgue convergence theorem.
\end{proof}

\begin{lemma} \label{lem:lipschitz_nemytskij_der_A}
    Let \cref{ass:risk_ingred} hold. Let $D_xA\colon L^\infty\big([0,T], L^{p_1}_\mu(\Theta,\R^n) \big) \to L^\infty\big([0,T], \mathcal Y \big)$ be the Nemytskij operator defined in \cref{eq:def_nemytskij_der_A}, where $\mathcal Y \coloneqq \mathcal L \left( L_\mu^{p_1}(\Theta,\R^n); L_\mu^{1}(\Theta,\R) \right)$.
    Then, $D_x A$ is locally Lipschitz continuous, i.e., for every $R>0$ there exists a constant $L_{D_x A}(R)>0$ (depending on $R$) such that 
    \begin{equation*}
        \sup_{t\in [0,T]} \| D_x A(Z)_t - D_x A(Z')_t \|_{\mathcal Y} \leq L_{D_x A}(R) \sup_{t\in [0,T]} \| Z_t - Z'_t \|_{L^{p_1}_\mu(\Theta,\R^n)},
    \end{equation*}
    whenever $\max \Big\{ \sup_{t\in [0,T]} \| Z_t \|_{L^{p_1}_\mu(\Theta,\R^n)}, \sup_{t\in [0,T]} \| Z_t' \|_{L^{p_1}_\mu(\Theta,\R^n)} \Big\} \leq R$.
\end{lemma}
\begin{proof}
    Let us consider $Z,Z' \in L^\infty\big([0,T], L^{p_1}_\mu(\Theta,\R^n) \big)$. Then, using \cref{ass:risk_ingred}, we compute
    \begin{equation} \label{eq:lipsch_cont_der_A_1}
        \begin{split}
            \sup_{t\in [0,T]} \| &D_x A(Z)_t  - D_x A(Z')_t \|_{\mathcal Y} =
            \sup_{t\in [0,T]} \sup_{\|\zeta\|_{L_\mu^{p_1}}\leq 1} \| D_x A(Z)_t [\zeta] - D_x A(Z')_t[\zeta] \|_{L_\mu^1} \\
            &\leq \sup_{t\in [0,T]} \sup_{\|\zeta\|_{L_\mu^{p_1}}\leq 1}  \int_{\Theta}
            \big| D_x a\big(t, Z_t(\theta), \theta \big) - D_x a\big(t, Z_t'(\theta), \theta \big) \big| |\zeta(\theta)| \,\diff \mu(\theta) \\
            & \leq \sup_{t\in [0,T]} \sup_{\|\zeta\|_{L_\mu^{p_1}}\leq 1} \int_{\Theta}
            \Big( L_{\mathrm{der}}\big(|Z_t(\theta)|,\theta \big) + L_{\mathrm{der}} \big( |Z_t'(\theta)|, \theta \big) \Big) |Z_t(\theta) - Z_t'(\theta)|  |\zeta(\theta)|
            \,\diff \mu(\theta).
        \end{split}
    \end{equation}
Then, we observe that, for every $\zeta \in L_\mu^{p_1}(\Theta,\R^n)$, we have that the map $\theta \mapsto \Big( L_{\mathrm{der}}\big(|Z_t(\theta)|,\theta \big) + L_{\mathrm{der}} \big( |Z_t'(\theta)|, \theta \big) \Big) |\zeta(\theta)| $ belongs to $L_\mu^{p_1'}(\Theta,\R)$, where $p_1' =p_1/(p_1-1)$ is the conjugate exponent of $p_1$. Indeed, recalling that by virtue of \cref{ass:risk_ingred} $L_{\mathrm{der}}(r,\theta)\coloneqq C^{\mathrm{Lip}}_{\mathrm{der}} r^{p_1-2} + g^{\mathrm{Lip}}_{\mathrm{der}}(\theta)$ with $g^{\mathrm{Lip}}_{\mathrm{der}} \in L_\mu^{\frac{p_1}{p_1-2}}(\Theta,\R)$, so that $L_{\mathrm{der}}(|Z_t(\cdot)| ,\cdot), L_{\mathrm{der}}(|Z_t'(\cdot)|, \cdot) \in L_\mu^{\frac{p_1}{p_1-2}}(\Theta,\R)$. Finally, we observe that $(p_1-1)'=(p_1-1)/(p_1-2)$.
Hence, leveraging on the H\"older inequality, we deduce that
\begin{equation} \label{eq:lipsch_cont_der_A_2}
    \begin{split}
        \int_\Theta 
        \Big| L_{\mathrm{der}}\big(|Z_t(\theta)|,\theta \big) &+ L_{\mathrm{der}} \big( |Z_t'(\theta)|, \theta \big) \Big|^{\frac{p_1}{p_1-1}} |\zeta(\theta)|^{\frac{p_1}{p_1-1}}
        \, \diff \mu(\theta) \\
        &
        \leq 
        \left\|   \big| L_{\mathrm{der}}\big(|Z_t(\cdot)|,\cdot \big) + L_{\mathrm{der}} \big( |Z_t'(\cdot)|, \cdot \big) \big|^{\frac{p_1}{p_1-1}} \right\|_{L_\mu^{\frac{p_1-1}{p_1-2}}} 
        \left\| |\zeta(\cdot)|^{\frac{p_1}{p_1-1}} \right\|_{L_\mu^{p_1-1}} \\
        & = \left\|   L_{\mathrm{der}}\big(|Z_t(\cdot)|,\cdot \big) + L_{\mathrm{der}} \big( |Z_t'(\cdot)|, \cdot \big) 
        \right\|_{L_\mu^{\frac{p_1}{p_1-2}}}^{\frac{p_1}{p_1-1}}
        \left\| \zeta(\cdot) \right\|_{L_\mu^{p_1}}^{\frac{p_1}{p_1-1}}
    \end{split}
\end{equation}
Hence, by combining \cref{eq:lipsch_cont_der_A_1,eq:lipsch_cont_der_A_2}, we obtain that
\begin{equation*}
    \begin{split}
    \sup_{t\in [0,T]} & \| D_x A(Z)_t  - D_x A(Z')_t \|_{\mathcal Y}\\
    & \leq
    \sup_{t\in [0,T]} \left\|   L_{\mathrm{der}}\big(|Z_t(\cdot)|,\cdot \big) + L_{\mathrm{der}} \big( |Z_t'(\cdot)|, \cdot \big) 
        \right\|_{L_\mu^{\frac{p_1}{p_1-2}}}^{\frac{p_1}{p_1-1}} \cdot 
        \sup_{t\in [0,T]} \| Z_t -Z_t' \|_{L_\mu^{p_1}} \\
    & \leq \left( \sup_{t\in [0,T]} \left\|   L_{\mathrm{der}}\big(|Z_t(\cdot)|,\cdot \big) \right\|_{L_\mu^{\frac{p_1}{p_1-2}}} + \sup_{t\in [0,T]} \left\| L_{\mathrm{der}} \big( |Z_t'(\cdot)|, \cdot \big) 
        \right\|_{L_\mu^{\frac{p_1}{p_1-2}}} \right)^{\frac{p_1}{p_1-1}}
        \sup_{t\in [0,T]} \| Z_t -Z_t' \|_{L_\mu^{p_1}}.
    \end{split}
\end{equation*}
To conclude, we observe that
\begin{equation*}
    \begin{split}
    \sup_{t\in [0,T]} \left\|   L_{\mathrm{der}}\big(|Z_t(\cdot)|,\cdot \big) \right\|_{L_\mu^{\frac{p_1}{p_1-2}}} &\leq
    \left\|
    g^{\mathrm{Lip}}_{\mathrm{der}}
    \right\|_{L_\mu^{\frac{p_1}{p_1-2}}} + C^{\mathrm{Lip}}_{\mathrm{der}}  \sup_{t\in [0,T]}
    \left\| |Z_t(\cdot)|^{p_1-2}
    \right\|_{L_\mu^{\frac{p_1}{p_1-2}}}\\
    &\leq \left\|
    g^{\mathrm{Lip}}_{\mathrm{der}}
    \right\|_{L_\mu^{\frac{p_1}{p_1-2}}} + C^{\mathrm{Lip}}_{\mathrm{der}}  \sup_{t\in [0,T]}
    \left\| Z_t
    \right\|_{L_\mu^{p_1}}^{p_1-2},
    \end{split}
\end{equation*}
and the same estimate holds when replacing $Z_t$ with $Z'_t$.
\end{proof}

\bibliographystyle{abbrv}
\bibliography{references}

\end{document}